\documentclass{amsart}
\usepackage{amscd}
\usepackage{amssymb}
\usepackage{cite}
\usepackage{amsmath}
\usepackage{comment}
\usepackage{bookmark}
\usepackage{mathtools}
\usepackage{enumerate}
\usepackage{diagbox}
\usepackage{caption}
\usepackage[usenames, dvipsnames]{color}
\usepackage{color}
\usepackage{hyperref}					
\hypersetup{colorlinks,
	linkcolor=blue,%
	citecolor=blue}

  \usepackage{mathrsfs}
\newtheorem{theorem}{Theorem}[section]
\newtheorem{thm}[theorem]{Theorem}
\newtheorem{lem}[theorem]{Lemma}
\newtheorem{definition}[theorem]{Definition}

\newtheorem{defi}[theorem]{Definition}
\newtheorem{cor}[theorem]{Corollary}
\newtheorem{corollary}[theorem]{Corollary}

\newtheorem{prop}[theorem]{Proposition}

  \usepackage{makecell}
\newcommand{\norm}[1]{\left\lVert#1\right\rVert}

\usepackage{tikz}
\numberwithin{equation}{section}
\newcommand{\cA}{{\mathcal A}}

\newcommand{\cE}{{\mathcal E}}

\newcommand{\cH}{{\mathcal H}}

\newcommand{\cM}{{\mathcal M}}

\newcommand{\cP}{{\mathcal P}}

\newcommand{\cZ}{{\mathcal Z}}

 \usepackage{color}
 
\usepackage{mathtools}
\usepackage{hyperref}                   
\hypersetup{colorlinks,
    linkcolor=blue,%
    citecolor=blue}
\usepackage{etoolbox}
\begin{document}

\title[Lack of isomorphic embeddings of $\ell_{p,q}$  into $L_{p,q}(\mathcal{M},\tau)$]{Lack of isomorphic embeddings of $\ell_{p,q}$  into $L_{p,q}(\mathcal{M},\tau)$
over a noncommutative probability space
}

\author[J. Huang]{Jinghao Huang}
\address{Institute for  Advanced Study in  Mathematics of HIT, Harbin Institute of Technology, Harbin, 150001, China}
\email{{\color{blue}jinghao.huang@hit.edu.cn}}

\author[O. Sadovskaya]{Olga Sadovskaya}
\address{Institute of Mathematics, Tashkent, 100084, Uzbekistan.}
\email{\color{blue}sadovskaya-o@inbox.ru}

\author[F. Sukochev]{Fedor Sukochev}
\address{School of Mathematics and Statistics, University of NSW, Sydney,  2052, Australia}
\email{\color{blue}f.sukochev@unsw.edu.au}

\author[D. Zanin]{Dmitriy Zanin}
\address{School of Mathematics and Statistics, University of NSW, Sydney,  2052, Australia}
\email{\color{blue}d.zanin@unsw.edu.au}

\thanks{J. Huang was supported the NNSF of China (No.12031004 and 12301160). F. Sukochev and D. Zanin were supported by the ARC}


\begin{abstract} 
 We prove that the sequence space $\ell_{p,q}$ does not embed into $L_{p,q}(\cM,\tau)$ for any noncommutative probability space $(\cM,\tau)$, $1< p<\infty $, $1\le q<\infty$, $p\ne q$. 
Several applications to the isomorphic classification of noncommutative $L_{p,q}$-spaces are given, which extend and complement  earlier results in   \cite{KS,SS,HS21,HRS,HRS00}.
\end{abstract}

\maketitle

\section{Introduction}

The Lorentz spaces $L_{p,q}$ are important generalizations of $L_p$-spaces which 
 were introduced by G.G. Lorentz in \cite{Lorentz50,Lorentz51}. 
 Their importance has been demonstrated in several areas of analysis such as harmonic analysis, interpolation theory, etc. (see e.g. \cite{BS,Dilworth,CD85} and references therein).
Recall that, 
for 
  a measure space  $(\Omega,\Sigma, \mu)$ and  $1<p <\infty$ and $1\le q<\infty$, the Lorentz space $L_{p,q}(\Omega)$ is the space  of all measurable functions $f$ on $\Omega$ such that
$$ \norm{f}_{p,q}:= \left(  \int_{0}^\infty \mu(t;f ) ^q d t^{q/p} \right)^{1/q}<\infty ,$$
where $\mu(f )$ stands for  the decreasing rearrangement of $|f|$ (see the next section).

The study of subspaces of $L_{p,q}$-spaces has gained  attention during past decades.
Recall that a Banach space $X$ is said to be primary if whenever $X$ isomorphic to $Y\oplus Z$, then either $Y$ or $Z$ is isomorphic to $X$. It is known that $L_{p,q}(0,1)$ and $L_{p,q}(0,\infty)$ are primary spaces  (see \cite[Theorem  2.d.11.]{LT2} and \cite[Theorem A.1]{Dilworth90}).
There are several criteria for  a sequence in $L_{p,q}$ to be equivalent to the $\ell_q$-basis (see e.g. \cite[Lemma 2.1]{CD} and \cite[Proposition 4.e.3]{LT1}). 
This together with a  criterion for a symmetric sequence    generating a complemented subspace in a Banach lattice 
in \cite[Lemma 8.10]{JMST} shows  that  $L_{p,q}(0,1)$ and  $L_{p,q}(0,\infty)$ are not isomorphic to each other,  $1< p<\infty$, $1\le q<\infty$, $p\ne  q $, and $\ell_{p,q}$ 
is not isomorphic to a complemented subspace of $L_{p,q}(0,1)$
 \cite[Corollary 2.2]{CD} (see also \cite{Dilworth}).
A stronger result showing that $\ell_{p,q}$ does not embed isomorphically into $L_{p,q}(0,1)$ has been proven in \cite{KS} and \cite{SS} by the so-called ``subsequence splitting lemma'' introduced and studied in \cite{S96} (see also \cite{DDS07,DSS,ASS,HSS,SS2004} for various applications of this techniques ).
It is interesting to note that there are symmetric function  spaces on $(0,1)$ which contain isomorphic copies of the space $\ell_{p,q}$, $1<p<2$, $1\le q<\infty$ \cite{Novikova}.
It is also known \cite{KS,SS} that
$L_{p,q}(0,\infty)$ does not embed isomorphically into the space $L_{p,q}(0,1)\oplus \ell_{p,q}$.
Hence,
 if $1<p<\infty$, $1\le q<\infty $,  $p\neq q$,  then
$$\ell_{p,q}^n, ~n=1,2,\cdots, ~\ell_{p,q},\; L_{p,q}(0,1),\; L_{p,q}(0,1)\oplus \ell_{p,q}\;\;\mbox{and}\;\; L_{p,q}(0,\infty)$$
is the full  list of pairwise non-isomorphic $L_{p,q}$-spaces over a resonant measure space. 
For embedding of $\ell_r$ into $L_{p,q}$, we refer to \cite{AS}, \cite{CD89}, \cite{CF},  \cite[Theorems 7 and 10]{Dilworth},   \cite[Example 3.5]{JSZ} and \cite{RS}.
For the isomorphic classification of $L_p$-spaces (i.e., $\ell_p^n$,  $n=1,2,\cdots,$ $\ell_p$, $L_p(0,1)$), we refer to  \cite[Part III]{Wojtaszczyk} and
\cite[Chapter XII]{Banach}; 
 for the isomorphic classification of  weak $L_p$-spaces, we refer to
  \cite{Leung,Leung2,Leung3,LS}.
  The situations with classification of $L_p$-spaces and that of  weak $L_p$-spaces are completely different from that of $L_{p,q}$. 

The situation with isomorphic classification of $L_{p,q}$-spaces over arbitrary measure spaces (not necessarily resonant) is more involved.
If the atoms $A_n$'s in a $\sigma$-finite atomic measure space $\Omega$ satisfy the conditions  that $$\mu(A_n)\to 0 \mbox{ and }  \sum_{n \ge 0} \mu(A_n)=\infty  ,$$ then the $L_{p,q}$-space on such a measure space is 
denoted by $U_{p,q}$. 
  Johnson et.al \cite[p.31]{JMST} introduced  these spaces, as
  natural generalizations  of   Rosenthal's space $X_p$  \cite{Rosenthal}.
It is proved in \cite[Theorem 8.7]{JMST} (see also \cite[Proposition 2.f.7]{LT2}) that
   $U_{p,q}$  does not depend on the particular sequence $\{A_n\}$ used in the definition (up to an isomorphism).

 The following tree (Hasse
diagram in the terminology from \cite{HRS}) was obtained in \cite{HS21}. 
  For any spaces $X\ne Y$   listed in the tree below, $X$ is isomorphic to a subspace (indeed,   a complemented subspace) of $Y$ if and only  if $X$ can be joined to $Y$ through a descending branch.
Note that there exists atomic measure spaces $\Omega$ such that $L_{p,q}(\Omega)$ is isomorphic to 
$(\oplus_{n=1}^\infty \ell_{p,q}^n )_q$\cite[Theorem 4.19]{HS21}. 

\begin{center}
\begin{tikzpicture}
  \node (a) at (0,2) {$(\oplus_{n=1}^\infty \ell_{p,q}^n )_q$};
  \node (b) at (-2,0) {$L_{p,q}(0,1)$};
  \node (c) at (1,1) {$\ell_{p,q}$};
  \node (d) at (2,0) {$U_{p,q}$};
  \node (e) at (-1,-1) {$L_{p,q}(0,1)\oplus \ell_{p,q}$};
  \node (f) at (0,-2) {$L_{p,q}(0,1)\oplus U_{p,q}$};
  \node (min) at (0,-3) {$L_{p,q}(0,\infty )$};
  \draw  (a) -- (c)
    (a) -- (b)
     (e) -- (c)--(d)
     (b) --(e) --(f) -- (min)
    (d)-- (f)
  ;
\end{tikzpicture}
\end{center}
Embeddings of noncommutative $L_{p,q}$-spaces were studied in \cite{AHS}, which shows that function space $L_{p,q}(0,1)$ does not embed into operator ideal $C_{p,q}$ for some values of $(p,q)$,  thus extending the deep pioneering results of Arazy and Lindenstrauss \cite{AL} for $L_p$-spaces.
Also, by a result due to Gillespie \cite[Theorem 2.6]{Gillespie}, 
$C_{p,q}$ does not embed into $L_{p,q}(0,\infty )$,
complementing and extending pioneering results due to McCarthy for $C_p$-spaces\cite{Mc}.

 The isomorphic classification of noncommutative $L_p$-spaces was studied in \cite{S01,S96,HRS00,HRS,S00,SC}. 
The 
 main result of the present paper    is the following theorem, which extends the main result of \cite{KS,SS,HRS,HRS00} to the setting of $L_{p,q}(\mathcal{M},\tau)$ over a noncommutative probability space. 
The proof relies on noncommutative Khinchine-type inequalities developed in \cite{S01,DPPS,PS}. 
 Note that there exists $(p,q)$, $p\ne q$, such that  $\ell_{p,q}$ is isomorphic to a subspace of the predual of some noncommutative probability space \cite[Example 5]{HJSZ}.
It is shown in \cite{Rand} that 
$\ell_p$ never embeds in $L_{p,q}(\cM,\tau)$ for any semifinite von Neumann algebra $\cM$ if $p\ne q$ and  $p\ne2$.

\begin{thm}\label{main theorem} 
Let $ (p,q)\in (1,\infty)\times [1,\infty )$. 
If $p\neq q,$ then $\ell_{p,q}$ does not embed isomorphically into $L_{p,q}(\mathcal{M},\tau)$ for an arbitrary noncommutative probability space $(\mathcal{M},\tau).$
\end{thm}

 We quote the following question \cite[Question 8]{JSZ}:
 \begin{quote}
   Suppose (an ideal of compact operators) $\mathcal I$ admits an isomorphic embedding into $\mathcal{L}_p(\cM,\sigma)$, $1\le p<2$, for some finite von Neumann algebra $(\cM,\sigma)$. Is it true that the commutative core $I$ of the ideal 
   $\mathcal{I}$ admits an isomorphic embedding into $L_p(0,1)$?
 \end{quote}
 If one considers  the question in  setting of $L_{p,q}(\cM,\tau)$,  $p>2$,  rather than $L_p(\cM,\tau)$ with $1\le p<2$, then we have an affirmative answer, see Theorem \ref{sequence_com_nc} below (the result holds for more general symmetric spaces), which proves   
  Theorem \ref{main theorem} for $p>2$. 
 The case for $1<p\le 2$ requires different techniques. 
  The difficulty in the study of subspace of  $L_{2,q}$ has been demonstrated in \cite{KS,SS} and \cite{CD89}.

 As an application of Theorem \ref{main theorem}, we obtain the following Hasse diagram of isomorphic embeddings in Section \ref{application} (see Corollaries \ref{cor5.3} and \ref{infinityinto}). As before, 
  for any spaces $X\ne Y$   listed in the tree below, $X$ is isomorphic to a subspace (indeed,   a complemented subspace) of $Y$ if and only  if $X$ can be joined to $Y$ through a descending branch.

%
%
{\tiny
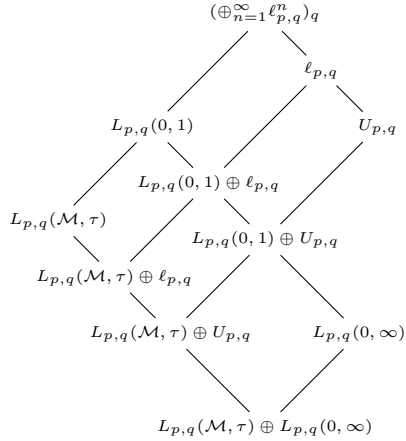
\begin{center}
\begin{tikzpicture}
  \node (a) at (0,1.5) {$(\oplus_{n=1}^\infty \ell_{p,q}^n )_q$};
  \node (b) at (-1.5,0) {$L_{p,q}(0,1)$};
  \node (c) at (0.75,0.75) {$\ell_{p,q}$};
  \node (d) at (1.5,0) {$U_{p,q}$};
  \node (e) at (-0.75,-0.75) {$L_{p,q}(0,1)\oplus \ell_{p,q}$};
  \node (f) at (0,-1.5) {$L_{p,q}(0,1)\oplus U_{p,q}$};
  \node (g) at (-2.75,-1.25)  {$L_{p,q}(\cM,\tau )$};
   \node (h) at (-2,-2)  {$L_{p,q}(\cM,\tau )\oplus \ell_{p,q}$};
      \node (i) at (-1.25,-2.75)  {$L_{p,q}(\cM,\tau )\oplus U_{p,q}$};
  \node (j) at (1.25,-2.75) {$L_{p,q}(0,\infty )$};
   \node (min) at (0,-4) {$L_{p,q}(\cM,\tau)\oplus L_{p,q}(0,\infty )$};
    \node (title1) at  (0,-4.5) {where $(\cM,\tau)$ is an arbitrary non-trivial   noncommutative probability space};
        \node (title2) at  (0,-4.8) { which is not of the form $\oplus_{1\le k <n}\mathbb{M}_k \otimes \cA_k$, $n<\infty$ ($\cA_k$ is a commutative algebra). };
  \draw  (a) -- (c)
    (a) -- (b)
     (e) -- (c)--(d)
     (b) --(e) --(f) -- (j)--(min)
    (d)-- (f)--(i)
    (b)--(g)
    (g)--(h)
    (e)--(h)--(i)--(min)
  ;
 
\end{tikzpicture}
\end{center}
 }

\section{Preliminaries}\label{prel}

\subsection{Symmetric function spaces}

Let $I$ be $(0,1)$ or $(0,\infty)$ and let $L(I)$ denote the space of all Lebesgue-measurable functions $x$ on $I.$
For a Lebesgue-measurable function $x$ on $I,$ we define its {\it distribution function} by the formula
$$d_x(s)=m(\{t:\ x(t)>s\}),\quad s\in\mathbb{R},$$
where $m$ stands for Lebesgue measure. 
Denote by $S(I)$ the subalgebra of $L(I)$
 consisting of all functions $f$ such that $d_{|f|}(s) < \infty$  for some 
 $s > 0$.
 
Two measurable functions $x$ and $y$ are called {\it equimeasurable} (written, $x\sim y$) if their distribution functions $d_x$ and $d_y$ coincide. In particular, for every measurable function $x\in S    (I),$ the function $|x|$ is equimeasurable with its {\it decreasing rearrangement} $\mu(x)$ defined by the formula
$$\mu(t;x):=\inf \{\tau\geq0:\ d_{|x|}(\tau)<t \},\quad t>0.$$
If $ x,y\in S(I),$ then $\mu(x)=\mu(y)$ if and only if $|x|$ and $|y|$ are equimeasurable.

\begin{defi} [see e.g. \cite{KPS}] Let $X\subset S(I)$ be a quasi-Banach space.
\begin{enumerate}[{\rm (a)}]
\item $X$ is said to be a quasi-Banach function space if, from $x\in X,$ $y\in S(I)$ and $|y|\leq |x|,$ it follows that $y\in X$ and $\left\|y\right\|_X\leq \left\|x\right\|_X.$
\item a quasi-Banach function space $X$ is said to be a symmetric if, for every $x\in X$ and any measurable function $y,$ the assumption $\mu(y)=\mu(x)$ implies that $y\in X$ and $\left\|y\right\|_X=\left\|x\right\|_X.$
\end{enumerate}
\end{defi}
Without lost of generality, in what follows we always assume that $\left\|\chi_{(0,1)}\right\|_X=1$ for any symmetric function space  $X=X(0,1)$ on $(0,1)$.

\subsection{Noncommutative symmetric spaces}For detailed exposition of material in this subsection,  we refer to \cite{DPS}.
In what follows,  $\cH$ is a  Hilbert space and $B(\cH)$ is the
$*$-algebra of all bounded linear operators on $\cH$ equipped with the uniform norm $\left\|\cdot\right\|_\infty$, and
$\mathbf{1}$ is the identity operator on $\cH$.
Let $\mathcal{M}$ be
a von Neumann algebra on $\cH$.
We denote by $\cP(\cM)$ the collection  of all projections in $\cM$, by $\cM'$ the commutant of $\cM$ and by $\cZ(\cM)$ the center of $\cM$.

A closed, densely defined operator $x:\mathfrak{D}\left( x\right) \rightarrow \cH $ with the domain $\mathfrak{D}\left( x\right) $ is said to be {\it affiliated} with $\mathcal{M}$
if $yx\subseteq xy$ for all $y\in \mathcal{M}^{\prime }$, where $\mathcal{M}^{\prime }$ is the commutant of $\mathcal{M}$.
A  closed,
densely defined
operator $x:\mathfrak{D}\left( x\right) \rightarrow \cH $ affiliated with $\cM $ is said to be
{\it measurable}  if  there exists a
sequence $\left\{ p_n\right\}_{n=1}^{\infty}\subset \cP\left(\mathcal{M}\right)$, such
that $p_n\uparrow \mathbf{1}$, $p_n(\cH)\subseteq\mathfrak{D}\left(x\right) $
and $\mathbf{1}-p_n$ is a finite projection (with respect to $\mathcal{M}$)
for all $n$.
 The collection of all measurable
operators with respect to $\mathcal{M}$ is denoted by $S\left(
\mathcal{M} \right) $, which is a unital $\ast $-algebra
with respect to strong sums and products (denoted simply by $x+y$ and $xy$ for all $x,y\in S\left( \mathcal{M%
}\right) $).

From now on, let $\mathcal{M}$ be a
semifinite von Neumann algebra equipped with a faithful normal
semifinite trace $\tau$.

An operator $x\in S\left( \mathcal{M}\right) $ is called $\tau$-measurable if there exists a sequence
$\left\{p_n\right\}_{n=1}^{\infty}$ in $P\left(\mathcal{M}\right)$ such that
$p_n\uparrow \mathbf{1}$, $p_n(\cH)\subseteq \mathfrak{D}\left(x\right)$ and
$\tau(\mathbf{1}-p_n)<\infty $ for all $n$.
The collection $S\left( \mathcal{M}, \tau\right)
$ of all $\tau $-measurable
operators is a unital $\ast $-subalgebra of $S\left(
\mathcal{M}\right) $.

Consider the algebra $\mathcal{M}=L^\infty(0,\infty)$ of all
Lebesgue measurable essentially bounded functions on $(0,\infty)$.
The algebra $\mathcal{M}$ can be seen as an abelian von Neumann
algebra acting via multiplication on the Hilbert space
$\mathcal{H}=L^2(0,\infty)$, with the trace given by integration
with respect to Lebesgue measure $m.$
It is easy to see that the
algebra of all $\tau$-measurable operators
affiliated with $\mathcal{M}$ can be identified with
the algebra $S(0,\infty)$.

\begin{definition}\label{mu}
Let $x\in
S(\mathcal{M},\tau)$. The generalized singular value function $\mu(x):t\mapsto  \mu(t;x)$ of
the operator $x$ is defined by setting
$$
\mu(s;x)
=
\inf\{\left\|xp\right\|_\infty:\ p\in \cP(\cM)\mbox{ with}\ \tau(\mathbf{1}-p)\leq s\}, ~ \forall s\in (0,\infty ).
$$
\end{definition}


\begin{definition}\label{def:symmetric}
 A linear subspace $E$ of $S(\cM,\tau)$ equipped with a complete norm $\norm{\cdot}_E$, is called a symmetric space (of $\tau$-measurable operators) if $x\in S(\cM,\tau)$, $y \in E$ and $\mu(x)\le \mu(y)$ imply that $x\in E$ and $\norm{x}_E \le \norm{y}_E$.
\end{definition}

It is well-known that any symmetric space $E$ is a normed $\cM$-bimodule, that is, $axb\in E$ for any $x\in E$, $a,b\in \cM$ and $\left\|axb\right\|_E\leq \|a\|_\infty\left\|b\right\|_\infty \left\|x\right\|_E$ \cite{DP2,DPS}.
%

%
%
%
%
%

 A wide class of symmetric operator spaces associated with the von Neumman algebra $\cM$ can be constructed from concrete symmetric function spaces studied extensively in e.g. \cite{KPS}. Let 
 $\cE :=  E(0,\infty) $ be a symmetric function space on the semi-axis $(0,\infty)$ (or $\cE : =E(0,1) $ for a noncommutative probability space $(\cM,\tau)$, i.e., $\tau$ is a faithful normal tracial state). Then the pair 
 $$E(\cM,\tau)=\{x\in S(\cM,\tau):\mu(x)\in \cE \},\quad \left\|x\right\|_{E(\cM,\tau)}:=\left\|\mu(X)\right\|_{\cE }$$ is a symmetric space on $\cM$ \cite{Kalton_S} (see also \cite{DPS,LSZ}). 
%
%


We write $B\prec\prec A$ (and say that $B$ is submajorized by $A$ in the sense of Hardy--Littlewood--P\'{o}lya) if
$$\int_0^t\mu(s,B)ds\leq\int_0^t\mu(s,A)ds,\quad t>0.$$
If $  A,B\in L_1(\mathcal{M},\tau)$ are positive  operators  such that $B\prec\prec A$ and $\tau(B)=\tau(A),$ then we write $B\prec A$ (and say that $B$ is majorized by $A$ in the sense of Hardy--Littlewood--P\'{o}lya).

Recall  the following properties of submajorization (see e.g. \cite[Theorem  3.3.3 and Lemma 3.3.7]{LSZ})
\begin{equation}\label{maj property1}
A+B\prec\prec\mu(A)+\mu(B),\quad A,B\in (L_1+L_{\infty})(\mathcal{M},\tau)
\end{equation}
and 
\begin{equation}\label{maj property2}
A\oplus B\prec\prec A+B,\quad 0\leq A,B\in (L_1+L_{\infty})(\mathcal{M},\tau).
\end{equation}

\subsection{$L_{p,q}$ and $\ell_{p,q}$ spaces}

We set
$$L_{p,q}(I)=\Big\{ x\in S(I):\ \int_0^1\mu^q(s;x)ds^{\frac{q}{p}}<\infty\Big\}.$$
It becomes quasi-Banach space when equipped with the quasi-norm
$$\left\|x\right\|_{p,q}=\Big(\int_0^1\mu^q(s;x)ds^{\frac{q}{p}}\Big)^{\frac1q},\quad x\in L_{p,q}(I).$$
This quasi-norm is a norm for $1\le q\leq p$ and is equivalent to a norm for $q>p>1$\cite{BS}.

Similarly, we define sequence space $\ell_{p,q}$ by setting 
\begin{align}\label{lpqsequence}
\ell_{p,q}=\left\{x\in \ell_{\infty}: \sum_{k\geq0}\mu(k;x)^q(k+1)^{\frac{q}{p}-1}<\infty\right\}.
\end{align}
It becomes quasi-Banach space when equipped with the quasi-norm
$$\left\|x\right\|_{\ell_{p,q}}=\Big(\sum_{k\geq0}\mu(k;x)^q (k+1)^{\frac{q}{p}-1}\Big)^{\frac1q},\quad x\in \ell_{p,q}.$$
Again, this quasi-norm is a norm for  $1\le q\leq p$ and is equivalent to a norm for $q>p>1$.
For detailed exposition of $L_{p,q}$-spaces,  see \cite{BS} and \cite{Dilworth}.

We recall below a well-known result (see e.g. \cite[Lemma 2.1]{CD}, see also  \cite{ACL, Dilworth,KamMal,LT1,Tradacete}).
\begin{lem}\label{2.1}
Suppose that $0<p,q<\infty$.
Let $\left\{f_n\right\}_{n\ge 1}$ be a sequence of unit vectors in $L_{p,q}(0,\infty)$ such that $f^*_n \to 0$ pointwise as $n \to \infty $.
Then,
there exists a  subsequence of $\{f_n\}_{n\ge 1}$ which
is equivalent to the unit vector basis of $\ell_q$.
\end{lem}

\subsection{Khinchine-type inequalities}
Let $\{r_k\}_{k\ge 1}$ be the Rademacher functions on $(0,1)$. 
\begin{thm}\label{vector vs tensor} Let $1<p<\infty$ and $1\leq q<\infty$ and let $A_k\in L_{p,q}(\mathcal{M},\tau),$ $k\geq 1 .$ The exists constant $c_{p,q}>0$ depending on $p,q$ only such that 
\begin{enumerate}[{\rm (a)}]
\item if $q\leq p,$ then
$$\inf_{t\in(0,1)}\left\|\sum_{k\geq 1}r_k(t)A_k\right\|_{L_{p,q}(\cM,\tau)}\leq c_{p,q}\left\|\sum_{k\geq 1}r_k\otimes A_k\right\|_{L_{p,q}(L_\infty \overline{\otimes} \cM )} $$
and 
$$\int_0^1\left\|\sum_{k\geq 1}r_k(t)A_k\right\|_{L_{p,q}(\cM,\tau)} dt\leq c_{p,q}\left\|\sum_{k\geq 1}r_k\otimes A_k\right\|_{L_{p,q}(L_\infty \overline{\otimes} \cM )};$$
\item if $q\geq p,$ then
$$\left\|\sum_{k\geq 1 }r_k\otimes A_k\right\|_{L_{p,q}(L_\infty \overline{\otimes} \cM )}\leq c_{p,q}\sup_{t\in(0,1)}\left\|\sum_{k\geq 1 }r_k(t)A_k\right\|_{L_{p,q}(\cM,\tau)}.$$
\end{enumerate}
\end{thm}
\begin{proof} 
Note that if $q\le p$, then $L_{p,q}$  is $D^*$-convex, see \cite[Appendix A(d)]{PS} or \cite[Proposition 2.5]{Spos98}.
By \cite[Proposition 2.5]{Spos98},  we have 
\begin{align*}
\int_0^1  \norm{ \sum_{k=1}^{ n} r_k (t) A_k    }_{L_{p,q}(\cM,\tau)}  dt & = \int_0^1  \norm{ \sum_{i=1}^{2^{n}}  \chi_{i}^{(n)}(t ) \left( \sum_{k=1}^n (-1)^{i_k}  A_k  \right  )   }_{L_{p,q}(\cM,\tau)}  dt \\
&\lesssim  \norm{ \sum_{i=1}^{2^{n}} \chi_{i}^{(n)}\otimes  \left( \sum_{k=1} ^{n }  (-1)^{i_k}  A_k   \right )   }_{L_{p,q}(\cM,\tau)}\\
&=\left\|\sum_{k=1}^{n}  r_k\otimes A_k\right\|_{L_{p,q}(L_\infty \overline{\otimes} \cM )}, 
\end{align*}
where $ \chi_i^{(n)}$ stands for the indicator function of the interval $[( i-1)\cdot 2^{-n}, i \cdot 2^{-n})$, $\sum_{k=1}^{n } i_k 2^{n-k  } =i-1$, $i_k=\pm 1$. 
This proves the first assertion. 

If $q\ge p$, then $q'\le p'$, where $\frac{1}{p}+\frac{1}{p'}=1$ and $\frac{1}{q}+\frac{1}{q'}=1$. 
By \cite[Appendix A(d)]{PS},   $L_{p',q'}$ is $D^*$-convex. 
Since 
 $L_{p,q}$ coincides with the dual of $L_{p',q'}$, 
it follows from   \cite[Appendix A(a)]{PS} (or \cite[Proposition 2.3]{Spos98}) that 
  $L_{p,q}$  is $D$-convex.
By \cite[Appendix A(c)]{PS} (see also \cite[Propositions 2.3 and  2.5]{Spos98}), we have 
$$\norm{\sum_{k=1}^n r_k \otimes A_k }_{L_{p,q}(L_\infty (0,1)\overline{\otimes} \cM)}\lesssim \norm{\sum_{k=1}^n r_k \otimes A_k }_{L_\infty (L_\infty (0,1)\overline{\otimes} L_{p,q}(\cM,\tau))},$$
which completes the proof. 
\end{proof}
 
Recall that $L_{p,q}$ is $1$-convex,  $\max\{p,q\}+\varepsilon$-concave, $\varepsilon>0$\cite[Theorem 3.4]{Creekmore} and both Boyd indices of $L_{p,q}$ 
are equal to $p$\cite{BS}.
The following  results follow from  \cite[Theorems 4.1 and 4.11]{DPPS} (see also \cite[Theorem 1.1]{LS}).

\begin{thm}\label{lust-piquard} 
Let $\cM$ be a von Neumann algebra equipped with a   semifinite faithful normal trace $\tau.$
For   any finite sequence $(A_k)$ of self-adjoint operators in $L_{p,q}(\cM,\tau)$, we have 
$$\left\|\sum_{k\geq 1 }r_k\otimes A_k \right\|_{L_{p,q}(L_\infty \overline{\otimes} \cM )}\leq c_{p,q}\left\|\left(\sum_{k\geq 1}A_k^2\right )^{\frac12}\right\|_{L_{p,q}(\cM,\tau)},$$
where $c_{p,q}$ is a constant depending on $p$ and $q$ only. 
\end{thm}

We denote by $L_1(L_{p,q}(\cM,\tau))$ the Bochner space $L_1((0,1); L_{p,q}(\cM,\tau))$. 
\begin{thm}\label{vector lust-piquard} 
Let $\cM$ be a von Neumann algebra equipped with a   semifinite faithful normal trace $\tau.$
Let $ (p,q)\in (1,\infty)\times [1,\infty )$. 
For any finite sequence $(A_k)$ of self-adjoint operators in $L_{p,q}(\cM,\tau)$, we have 
$$\norm{\sum_{k\ge 1}r_k\otimes A_k}_{L_1(L_{p,q}(\cM,\tau))} =\int_0^1\left\|\sum_{k\geq 1 }r_k(t)A_k\right\|_{L_{p,q}(\cM,\tau)} dt\leq c_{p,q}\left\|\left(\sum_{k\geq 1 }A_k^2 \right )^{\frac12}\right\|_{L_{p,q}(\cM,\tau)}, $$
where $c_{p,q}$ is a constant depending on $p$ and $q$ only. 
\end{thm}

\subsection{Subsequence splitting property}\label{splitting}

 The subsequence splitting principle below  shows that in noncommutative $L_{p,q}$-spaces, each bounded sequence contains a subsequence
which is a perturbation of the sum of an equimeasurable sequence
and a sequence which is left and right disjointly supported.

A symmetric space $\cE\subset S(\cM,\tau)$ is said to have the \emph{Fatou property} if for every upwards directed net $\{a_\beta\}$ in $\cE^+$, satisfying $\sup_\beta \left\|a_\beta\right\|_\cE <\infty$, there exists an element $a\in \cE^+$ such that $a_\beta \uparrow a$ and $\left\|a\right\|_\cE = \sup_\beta\|a_\beta\|_\cE$. Examples such as Lorentz spaces, Orlicz spaces, etc. all have symmetric norms which have  Fatou property.

For the  case of non-selfadjoint operators, the proof of the following theorem was  given in \cite[Proposition 2.7]{DDS07}. For the special case of self-adjoint operators, 
the proof follows along the same lines by   replacing  \cite[Lemma 2.4]{DDS07} with \cite[Lemma 7.1]{DPS16} (see also \cite[Theorem 2.5]{CDS} or \cite[Lemma 7.2]{HNPS}). 
Note that if $E(0,\infty)$ has the Fatou property, then $E(\cM,\tau)$ has the Fatou property \cite{DP2,DPS}. 

\begin{thm}\label{subsequence splitting property}
Let $\cM$ be an atomless von Neumann equipped with a semifinite faithful normal trace $\tau$. 
Let $E(0,\infty )$ be a separable  symmetric space having the Fatou property. 
Let $\{x_k\}_{k\geq0}\subset E(\mathcal{M},\tau)$ be a bounded  sequence of self-adjoint operators. 
There exists a subsequence $\{x'_k\}_{k\geq0}$ of $\{x_k\}_{k\geq0}$ such that 
\begin{enumerate}[{\rm (a)}]
\item $x_k'=u_k+v_k+w_k,$ for all $k\geq0;$
\item every $u_k,v_k,w_k\in E(\mathcal{M},\tau),$ $k\geq0,$ is self-adjoint;
\item elements of the sequence $\{\mu(u_k)\}_{k\geq0}$ are equimeasurable;
\item elements of the sequence $\{v_k\}_{k\geq0}$ are pairwise orthogonal, i.e. $v_kv_l=0,$ $k\neq l;$
\item $w_k\to0$ in $E(\mathcal{M},\tau)$ as $k\to\infty.$
\end{enumerate}
If, in addition, that $\{x_k\}_{k\geq0}$ is weakly null, then $\{u_k\}_{k\geq0}$ and $\{v_k\}_{k\geq0}$ may be chosen to be weakly null as well. 
\end{thm}

 Recall  that two Banach spaces $X$ and $Y$ are said to be isomorphic (denoted by $X\approx Y$) if there exists an invertible bounded linear  operator from $X$ onto $Y$. Otherwise, we write $X\not \approx   Y$.
We say that $Y$
embeds into
$X$
isomorphically if there is a linear subspace
$Z\subset X$,
such that $Z$
is a Banach space and $Z$
is isomorphic to $Y$.
If   $X$ is isomorphic to a subspace (resp. a complemented subspace) of  $Y$,
we write $X\hookrightarrow Y$ (resp. $X\xhookrightarrow{c } Y$).
If $X$ does not embed into $Y$
isomorphically as a subspace  (resp. a complemented subspace), we write $X\not \hookrightarrow Y$ (resp. $X\not \xhookrightarrow{c} Y$).
We  write  $A \lesssim B$ if $A\le  c B$ for some universal constant $c$, and $A\approx  B$ if $A\lesssim B$ and $B\lesssim A$ ($A\approx_c  B$ if $\frac1c B\le A\le c B$).

If two sequences $\{a_k\}$ and $\{b_k\}$
of real positive numbers satisfy the condition
$$
0 < \inf_k \frac{a_k}{b_k} \le \sup_k \frac{a_k}{b_k} <\infty,
$$
we then write $\{a_k\}\sim \{b_k\}$.
We use the same notation to denote the equivalence between basic sequences and this should not cause any confusion.
A sequence $\left\{x_k\right\}_X$
in a Banach space $X$
is called semi-normalized,
if there are scalars
$0<a<b<\infty$
such that $a\le \norm{x_n }\le b$,
$n \ge 1$.
We denote the unit vector basis of $\ell_q$ and
$\ell_{p,q}$
by
$\{e^{\ell_q}_k \}_{k=1}^\infty $
and $\{e_k^{\ell_{p,q}}\}_{k=1}^\infty$, respectively.

\section{Auxiliary results}

\subsection{Estimates for the case when  $p=2,$ $q<2$}
The goal of this subsection of to establish noncommutative versions of results in \cite[Section 3]{SS}.
Lemma \ref{corrected semenov log} below  is a generalization of  the claim made in  \cite[Lemma 3]{SS}. 
Even though Lemma \ref{corrected semenov log} is similar to \cite[Lemma 3]{SS}, \cite[Lemma 3]{SS} is not sufficient for our purpose. 
\begin{lem}\label{semenov log} Let $0 < \alpha\le \infty$.
If $0\le x\in L_{2,1}(0,\alpha) \cap L_\infty (0,\alpha )$ and $\inf_{t\in (0,1)}x(t) >0 $, then
\begin{align}\label{eq:semenov log}
\left\| x^{\frac12}\right\|_{L_{2,1}(0,\alpha)}\leq c_{abs}\left\|x\right\|_{L_1(0,\alpha) }^{\frac12}(1+ \left\|\log(x)\right\|_{L_\infty(0,\alpha) })^{\frac12},
\end{align}
where $c_{abc}>0$ is a constant not depending on $x$.    
\end{lem}
\begin{proof} 
Set
$$y=\sum_{l\in\mathbb{Z}}2^{-l}\chi_{\{ 2^{-l}< x\le 2^{1-l}\}}. $$
There exist integers $l_+$ and  $l_-$ with $l_-\le l_+$ such that 
$$y=\sum_{l=l_-}^{l_+}2^{-l}\chi_{\{ 2^{-l}< x\le 2^{1-l}\}}  $$
and the measurable sets
$$ \{ 2^{-l_-}< x\le 2^{1-l_-}\}, \{ 2^{-l_+}< x\le 2^{1-l_+}\}   $$
are  not empty. 
In particular, we have 
$$\mu(y)=\sum_{l=l_-}^{l_+}2^{-l}\chi_{[d_x(2^{1-l}),d_x(2^{-l}))},$$
Since   $\alpha^{\frac12}-\beta^{\frac12}\leq(\alpha-\beta)^{\frac12}$ for any  $0<\beta<\alpha$ and 
$$\mu(y^{1/2})=\sum_{l=l_-}^{l_+}2^{-l/2}\chi_{(d_x(2^{1-l}),d_x(2^{-l}))},$$
it follows from the definition of $L_{p,q}$-norms that 
\begin{align*}
\left\| y^{\frac12}\right\|_{L_{2,1}(0,\alpha)}& 
=\sum_{l=l_-}^{l_+}2^{-\frac{l}2}\left(d_x(2^{-l})^{\frac12}-d_x(2^{1-l})  ^{\frac12}\right ) \leq \sum_{l=l_-}^{l_+}2^{-\frac{l}2}\left( d_x(2^{-l})-d_x(2^{1-l})\right )^{\frac12}.
\end{align*}
By    Cauchy's inequality, for any positive numbers $a_1,\cdots ,a_N$, we have 
$$\sum_{k=0}^Na_k^{\frac12}\leq(N+1)^{\frac12}\left( \sum_{k=0}^Na_k\right)^{\frac12},$$
and therefore, 
\begin{align}\label{esy}
\left\|y^{\frac12}\right\|_{L_{2,1}(0,\alpha)}\leq(l_+-l_-+1)^{\frac12}\left(
\sum_{l=l_-}^{l_+}2^{-l}\left(d_x(2^{-l})-d_x(2^{1-l})\right)\right)^{\frac12}.
\end{align}
Observe that
$$\left\|y\right\|_{L_1(0,\alpha)} =\sum_{l=l_-}^{l_+}2^{-l}(d_x(2^{-l})-d_x(2^{1-l})).$$
Since 
$$ 2^{-1}\inf_{t\in (0,1)}x(t) \le 2^{-l_+} \le 2^{- l_-} \le \norm{x}_{L_\infty(0,\alpha)} ,   $$
it follows that 
$$  \log (2^{-1}\inf_{t\in (0,1)}x(t) )  \le   -l_+ \cdot \log 2  \le   -l_- \cdot  \log 2  \le   \log (\norm{x}_{L_\infty(0,\alpha)} ) .   $$
That is,  \begin{align}\label{l-+}
 \frac{1}{\log 2}   \log (2^{-1}\inf_{t\in (0,1)}x(t) )  \le   -l_+    \le   -l_-   \le \frac{1}{\log 2}     \log (\norm{x}_{L_\infty(0,\alpha)} ) .
 \end{align}
For any $t\in(0,\alpha)$, we have 
$$ \inf_{t\in (0,1)}x(t)\le x(t) \le \norm{x}_\infty ,  $$
and therefore, 
$$ \log (\inf_{t\in (0,1)}x(t)) = \inf_{(0,1)} \log( x(t)) \le \sup_{(0,1)}\log( x(t)) = \log \norm{x}_\infty. $$
Hence, $\norm{\log x   }_\infty =  \max \left\{ \left| \log (\inf_{t\in (0,1)}x(t))\right| , \left| \log \norm{x}_{L_\infty(0,\alpha)} \right| \right \}$. 
We have
\begin{align}\label{es+-}
|l_+|,|l_-|&\stackrel{\eqref{l-+}}{\le} \max\left\{ \left|\frac{1}{\log 2}   \log (2^{-1}\inf_{t\in (0,1)}x(t) )\right|, \left|\frac{1}{\log 2}   \log (\norm{x}_{L_\infty(0,\alpha)} )\right| \right\}\nonumber \\
&~\le \frac{1}{\log 2}   \norm{\log x}_{L_\infty(0,\alpha)}  .  
\end{align}
On the other hand, we have 
\begin{align}\label{x12y}
\left\|x^{\frac12}\right\|_{L_1(0,\alpha)}\leq 2^{\frac12}\left \|y^{\frac12}\right\|_{L_1(0,\alpha)} ~ \mbox{ and}\quad \left\|y\right\|_{L_1(0,\alpha)}\leq \left\|x\right\|_{L_1(0,\alpha)} .
\end{align}
Combining above inequalities, we have 
\begin{align*}
\left\| x^{\frac12}\right\|_{L_{2,1}(0,\alpha)}&\stackrel{\eqref{x12y}}{\leq} 2^{\frac12}\left \|y^{\frac12}\right\|_{ L_{2,1}(0,\alpha)}\\
&  \stackrel{\eqref{esy}}{\le} 
2^{\frac12} (l_+-l_-+1)^{\frac12}\left(
\sum_{l=l_-}^{l_+}2^{-l}\left(d_x(2^{-l})-d_x(2^{1-l})\right)\right)^{\frac12}  \\
&\stackrel{\eqref{es+-}}{\le} 2^{\frac12}  \left(   \frac{2}{\log 2}   \norm{\log x}_{L_\infty(0,\alpha )}+1\right) \norm{y}_{L_1(0,\alpha)}^{\frac12}\\
&\stackrel{\eqref{x12y}}{\leq} 2^{\frac12}  \left(   \frac{2}{\log 2}  \norm{\log x}_{L_\infty(0,\alpha )} +1\right) \norm{x}_{L_1(0,\alpha)}^{\frac12}. 
\end{align*}
This completes the proof of \eqref{eq:semenov log}.
\end{proof}
Denote by $s(x)$ the characteristic function of the support of a function $x$. 
\begin{lem}\label{corrected semenov log}
Let  $0\leq x\in L_{\infty}(0,1)$ be such that $\left\|x\right\|_{L_\infty(0,1)}\leq n,$ $\left\|s(x)\right\|_{L_1(0,1)}\leq (n+1)t$ and $\left\|x\right\|_{L_1(0,1)}\leq t\cdot \log(en)$, $0< t<\infty$. 
We have   $$\left\|x^{\frac12}\right\|_{L_{2,1}(0,1 )}\leq c_{abs}t^{\frac12}\log(en),$$
where $c_{abc}>0$ is a constant not depending on $x$.   
\end{lem}
\begin{proof} Set $y:=x\chi_{(\frac1n,\infty)}(x)$ and $z:=x\chi_{[0,\frac1n]}(x).$
Since $L_{2,1}(0,1 )$ is a Banach space, it follows from the triangular inequality that  
\begin{equation}\label{cseq1}
\left\| x^{\frac12} \right\|_{L_{2,1}(  0,1 )}\leq \left\|y^{\frac12}\right\|_{L_{2,1}(0,1)}+\left\|z^{\frac12}\right\|_{L_{2,1}(0,1)}.
\end{equation}

 It follows from Lemma \ref{semenov log} that
$$\left\|y^{\frac12}\right\|_{L_{2,1}(0,1)}=\left\|\mu(y)^{\frac12}\right\|_{L_{2,1}(0,1)}\leq c_{abs}\left\|\mu(y)\right\|_{L_1(0,1)}^{\frac12}(1+\left\|\log(\mu(y))\right\|_{L_{\infty}(0,1)})^{\frac12}.$$
Observe that 
$$\left\|\mu(y)\right\|_{L_1(0,1)}\leq \left\|x\right\|_{L_1(0,1)} \leq\log(en)\cdot t$$
and,  by $\frac1n \le  y\le n$, we have 
$$\left\|\log(\mu(y))\right\|_{L_{\infty}(0,1)}\leq \log(n).$$
We obtain that 
\begin{equation}\label{cseq2}
\left\|y^{\frac12}\right \|_{L_{2,1}(0,1)}\leq c_{abs}    \cdot \log^{1/2}(en)\cdot t^{\frac12}
\left( 1+\log n\right)^{1/2} \le c_{abs}
t^{\frac12}\log(en).
\end{equation}
On the other hand, we have 
\begin{align*}
\left\|z^{\frac12}\right\|_{L_{2,1}(0,1)}= \left\|\frac{1}{\sqrt{n}} s(z)\right\|_{L_{2,1}(0,1)}\le 
 n^{-\frac12}\left\|s(x)\right\|_{L_{2,1}(0,1)}& =  n^{-\frac12}  \int_{0}^{ } 1    d t^{1/2} \\
 &=n^{-\frac12}\cdot m(s(x)) ^{\frac12}. 
 \end{align*} 
By $m(s(x))= \left\|s(x)\right\|_{L_1(0,1 )} \leq (n+1)t$, we have
\begin{equation}\label{cseq3}
\left\|z^{\frac12}\right\|_{L_{2,1}(0,1 )}\leq n^{-\frac12}\cdot (n+1)^{\frac12}t^{\frac12}\leq 2t^{\frac12}.
\end{equation}
Combining \eqref{cseq1}, \eqref{cseq2} and \eqref{cseq3}, 
we have 
\begin{align*}
\norm{x^{1/2}}_{L_{2,1}(0,1 )}\le c_{abs}
t^{\frac12}\log(en) +2t^{\frac12}  &\le c_{abs}
t^{\frac12}\log(en) +2t^{\frac12}\log(en) \\
& =    (c_{abs}+2) t^{\frac12} \cdot \log(en). 
\end{align*}
 This completes the proof. 
\end{proof}

Let $\cM$ be an atomless von Neumann algebra equipped with a faithful normal trace $\tau$. 
Recall that there exist  $*$-homomorphisms from $S(0,1)$ into $S(\cM,\tau)$ which preserve  singular value functions, see e.g. \cite[Lemma 1.3]{CKS92} or \cite[Lemma 4.1]{CS94} (see
also \cite[Proposition 3.3]{PS07}). 
The following result is a noncommutative version of Lemma 6 in \cite{SS}. 
\begin{lem}\label{an bound noncomm} 
Let $\cM$ be an atomless von Neumann algebra equipped with a faithful normal tracial state $\tau$. 
For every $k\geq0,$ let $i_k:S(0,1)\to S(\mathcal{M},\tau)$ be a $*$-homomorphism which preserves singular value functions. Define operators $A_n$ from $L_{2,1}(0,1)$ to $L_{2,1}(L_{\infty}(0,1)\otimes\mathcal{M},dm\otimes\tau)$ by setting
\begin{align}\label{estiAn}
A_nx:=\sum_{k=0}^n(k+1)^{-\frac12}r_k\otimes i_k(x),\quad n\geq2.
\end{align}
We have $$\left\|A_n\right \|_{L_{2,1}(0,1)\to L_{2,1}(L_{\infty}(0,1)\otimes\mathcal{M},dm\otimes\tau) }\leq c_{abs}\log(e n),~n\geq2,$$
where $c_{abs}$ is a constant not depending on $n$. 
\end{lem}
\begin{proof} Let $E\subset(0,1)$ be a measurable subset and let $t:=m(E).$
Let 
$$X:=\sum_{k=0}^n\frac1{k+1}i_k(\chi_E)\in L_{2,1}(\mathcal{M},\tau) .$$ By Theorem \ref{lust-piquard} above,  we have
\begin{align*}
\left\|A_n(\chi_E)\right \|_{L_{2,1}(L_{\infty}(0,1)\otimes\mathcal{M},dm\otimes\tau)}&=\norm{\sum_{k=0}^n(k+1)^{-\frac12}r_k\otimes i_k(\chi_{E})}_{L_{2,1}(L_{\infty}(0,1)\otimes\mathcal{M},dm\otimes\tau)}\\
&\le  c_{abs}\left\|  \left( \sum_{k=0}^n\frac1{k+1}i_k(\chi_E) \right)^{1/2}\right\|_{L_{2,1}( \mathcal{M}, \tau)} \\
&= c_{abs}\left\|X^{\frac12}\right\|_{L_{2,1}( \mathcal{M}, \tau)} .
\end{align*}
By the triangular inequality, we have
$$\left\|X\right\|_{L_1(\cM,\tau)} \leq\sum_{k=0}^n\frac{t}{k+1}\leq 2t\log(en),\quad \left\|X\right\|_{L_\infty(\cM,\tau)}\leq \sum_{k=0}^n\frac1{k+1}\leq 2 n  $$
and 
$$\left\|s(X)\right\|_{L_1(\cM,\tau)}\stackrel{\mbox
{\tiny \cite[Lemma 4.1.4]{DPS}}}{\leq}\sum_{k=0}^{n }\left\|s\left (\frac1{k+1}i_k(\chi_E) \right )\right\|_{L_1(\cM,\tau)}=(n+1)t,$$
where $s(X)$ denotes the support of the self-adjoint operator $X$. 
By Lemma \ref{corrected semenov log}, there exists a constant $c_{abs}$ such that 
$$\left\|X^{\frac12}\right\|_{L_{2,1}(\cM,\tau)}=\left\|\mu(X)^{\frac12}\right\|_{L_{2,1}(0,1)}\leq c_{asb}\cdot t^{\frac12}\cdot \log(en).$$
It follows now that
$$\left\|A_n(\chi_E)\right \|_{L_{2,1}(L_{\infty}(0,1)\otimes\mathcal{M},dm\otimes\tau)} \leq c_{abs}t^{\frac12}\log(en)=c_{abs}\left\|\chi_E\right\|_{L_{2,1}(0,1)}  \log(e n).$$
The assertion follows now from Lemma II.5.2 in \cite{KPS} (see also \cite[Lemma 5]{SS} for a similar result).
\end{proof}
Below, we obtain the estimate of $A_n$ for $L_{2,q}$. 
\begin{lem}\label{an l2q noncomm} Let the operators $A_n,$ $n\geq2,$ be as in Lemma \ref{an bound noncomm}. If $1\leq q\leq 2,$ then
$$\left\|A_n \right\|_{L_{2,q}(0,1)\to L_{2,q}(L_{\infty}(0,1)\otimes\mathcal{M},dm\otimes\tau) }\leq c_{abs}\log^{\frac1q}(en),\quad n\geq2.$$
\end{lem}
\begin{proof} It follows from Theorem 5.2.4 in \cite{BL} that $[L_{2,1},L_2]_{\frac{2q-2}{q},q}=L_{2,q}$ (real method of interpolation). It follows from   \cite[Proposition 2.g.15]{LT2} and \cite[Theorem 7.3.6]{DPS} that
\begin{align*}
&\quad \left\|A_n\right\|_{L_{2,q}(0,1)\to L_{2,q}(L_{\infty}(0,1)\otimes\mathcal{M},dm\otimes\tau) }\\
&\leq \left\|A_n\right\|_{L_{2,1}(0,1)\to L_{2,1}(L_{\infty}(0,1)\otimes\mathcal{M},dm\otimes\tau) }^{\frac{2-q}{q}}\left\|A_n\right\|_{L_{2 }(0,1)\to L_{2 }(L_{\infty}(0,1)\otimes\mathcal{M},dm\otimes\tau) }^{\frac{2q-2}{q}}.
\end{align*}
Using Lemma \ref{an bound noncomm}  and the  inequality
$$\left\|A_n\right\|_{L_{2 }(0,1)\to L_{2 }(L_{\infty}(0,1)\otimes\mathcal{M},dm\otimes\tau) }=\Big(\sum_{k=0}^n\frac1{k+1}\Big)^{\frac12}\leq \left( 1+\log n \right)^{1/2}=  \log^{\frac12}(e n),$$
we infer that
\begin{align*}\left\|A_n\right\|_{L_{2,1}(0,1)\to L_{2,1}(L_{\infty}(0,1)\otimes\mathcal{M},dm\otimes\tau) } &\leq \Big(c_{abs}\log(e n)\Big)\cdot\Big( \log^{\frac12}(e n)\Big)^{\frac{2q-2}{q}}\\
&=c_{abs}^{\frac{2-q}{q}} \left( \log(e n)\right) ^{\frac1q}.
\end{align*}
This completes the proof.
\end{proof}

We end this subsection with a noncommutative version of \cite[Lemma 8]{SS}. 
\begin{lem}\label{anx bound noncomm} Let the operators $A_n$ be as in Lemma \ref{an bound noncomm}. For every $x\in L_{2,q}(0,1),$ $1\leq q<2,$ we have
$$\left\|A_nx\right\|_{L_{2,q}(L_{\infty}(0,1)\otimes\mathcal{M},dm\otimes\tau)}=o(\log^{\frac1q}(n)),\quad n\to\infty.$$
\end{lem}
\begin{proof} Without loss of generality, we may assume that 
$x\geq0.$ Fix $\epsilon>0.$ Since $L_{2,q}$ is separable, it follows that there exists $t>0$ with $\left\|\mu(x)\chi_{(0,t)}\right\|_{L_{2,q}(0,1)} \leq\epsilon$ (see e.g. \cite[p.118]{LT2} or \cite[Chapter 4, Section 5]{KPS}). Set $$x_1:=(x-\mu(t;x))_+ \mbox{ and } x_2:=\min\left\{ x,\mu(t;x) \right \}.$$ 
By the triangular inequality, we have
\begin{align*}
 &~\quad \left\|A_nx\right\|_{L_{2,q}(L_{\infty}(0,1)\otimes\mathcal{M},dm\otimes\tau)}\\&~\leq \left\|A_n(x_1) \right\|_{L_{2,q}(L_{\infty}(0,1)\otimes\mathcal{M},dm\otimes\tau)} +\left\|A_n(x_2)\right\|_{L_{2,q}(L_{\infty}(0,1)\otimes\mathcal{M},dm\otimes\tau)}.
\end{align*}
By Lemma \ref{an l2q noncomm}, we have
\begin{align*}\left\|A_n(x_1)\right\|_{L_{2,q}(L_{\infty}(0,1)\otimes\mathcal{M},dm\otimes\tau)}&\leq c_{abs}\log^{\frac1q}(e n)\left\|x_1\right\|_{L_{2,q}(0,1)}\\
&\leq 
c_{abs}\log^{\frac1q}(e n) \left\|\mu(x)\chi_{(0,t)}\right\|_{L_{2,q}(0,1)}  \\&
\le c_{abs}\cdot \epsilon\cdot \log^{\frac1q}(e n).
\end{align*}
By Theorem \ref{lust-piquard}, we have 
\begin{align*}
\left\|A_n(x_2)\right\|_{L_{2,q}(L_{\infty}(0,1)\otimes\mathcal{M},dm\otimes\tau)} &\leq c_{abs}\left\|\Big(\sum_{k=0}^n\frac1{k+1}(i_k(x_2))^2)^{\frac12}\right\|_{L_{2,q}( \mathcal{M}, \tau)} \\
&
 \leq c_{abs}\left\|\Big(\sum_{k=0}^n\frac1{k+1}(i_k(x_2))^2\Big)^{\frac12}\right\|_{L_{\infty }( \mathcal{M}, \tau)} \\
&
 \leq c_{abs}\left\| \sum_{k=0}^n\frac1{k+1}(i_k(x_2))^2 \right\|_{L_{\infty }( \mathcal{M}, \tau)}^{\frac12} \\
 &\le 
    c_{abs}\left( \ \sum_{k=0}^n \left\|  \frac1{k+1}(i_k(x_2))^2 \right\|_{L_{\infty }( \mathcal{M}, \tau)} \right ) ^{\frac12} \\
    &\le  c_{abs}\left( \sum_{k=0}^n \frac{1}{ k+1 } \norm{x_2}^2_{L_{\infty }(0,1)}\right)^{\frac12} \\ 
 &\leq 2c_{abs}\log^{\frac12}(n)\cdot \mu(t,x).
 \end{align*}
 Taking into account that $q <2$, we obtain that  
$$\limsup_{n\to\infty}\frac1{\log^{\frac1q}(n)}\left\|A_nx\right\|_{L_{2,q}(L_{\infty}(0,1)\otimes\mathcal{M},dm\otimes\tau)}    \leq  2  c_{abs}\epsilon.$$
Since $\epsilon>0$ is arbitrarily small, the assertion follows.
\end{proof}

\subsection{Estimates for the case when  $p<2$}
In this part, we collect some estimates which will be needed for the proof of Theorem \ref{main theorem} in the special  case when $p<2$. 
Note  that if $p<1$, then  $L_{p,q}$ is an interpolation space between $L_1$ and $L_{\frac{p}{2}}$\cite[Theorem 2]{Dilworth}. 
The following lemma is a consequence of \cite[Theorem 1.1.(b)]{CSZ}. 
\begin{lem}\label{inverse monotone norm}Let $p<1.$ If $x,y\in L_{\infty}(0,\infty)$ are finitely supported and if $y\prec x,$ then $$\left\|x\right\|_{L_{p,q}(0,\infty)}\leq c_{p,q}\left\|y\right\|_{L_{p,q}(0,\infty)}.$$
\end{lem}

\begin{lem}\label{bn norm} Let $1< p<2$ and $1\le q<\infty$ For every $k\geq0,$ let $i_k:S(0,1)\to S(\mathcal{M},\tau)$ be a $*$-homo\-morphism which preserves singular value functions. Define the operators $B_n$ from $L_{p,q}(0,1)$ to $L_1(L_{p,q}(\mathcal{M},\tau))$ by setting
$$(B_nx)(t):=\sum_{k=0}^nr_k(t)i_k(x),~ x\in  L_{p,q}(0,1), ~n\geq1.$$
We have $$\left\|B_n\right\|_{L_{p,q}(0,1)  \to L_1(L_{p,q}(\mathcal{M},\tau)) }\leq c_{p,q}n^{\frac1p},~n\geq1.$$
\end{lem}
\begin{proof} By Theorem \ref{vector lust-piquard}, we have
\begin{align*}
\left\|B_nx\right\|_{L_1(L_{p,q}(\mathcal{M},\tau))}&\lesssim\left\|
(\sum_{k=0}^n(i_k(x))^2)^{\frac12}\right\|_{L_{p,q}(\cM,\tau)}\\
&=\left( \int_0^\infty 
\mu\left(\sum_{k=0}^n(i_k(x))^2\right)^{\frac{q}{2}} d t^{q/p}\right)^{1/q }\\
&=\left( \left( \int_0^\infty 
\mu\left(  \sum_{k=0}^n(i_k(x))^2\right)^{\frac{q}{2}} d t^{\frac{q}{2}/\frac{p}{2}}\right)^{2/q } \right)^{\frac12}\\
&=\left\|\sum_{k=0}^n(i_k(x))^2\right\|_{ {L_{\frac{p}{2},\frac{q}{2}}(\cM,\tau)} }^{\frac12}.
\end{align*}

By Lemma \ref{inverse monotone norm}, the norm in $L_{\frac{p}{2},\frac{q}{2}}$ is anti-monotone with respect to   majorization (not submajorization). 
For every $N\geq0,$ it follows from \eqref{maj property2} that 
$$\sum_{k=0}^n(i_k(x\chi_{\{|x|\leq N\}}))^2\otimes e_k^{\ell_{p,q}}\prec \sum_{k=0}^n(i_k(x\chi_{\{|x|\leq N\}}))^2,$$
it follows that
$$\left\|
\sum_{k=0}^n(i_k(x\chi_{\{|x|\leq N\}}))^2\right \|_{L_{\frac{p}{2},\frac{q}{2}}(\cM,\tau)}\leq c_{p,q} \left\|\sum_{k=0}^n(i_k(x\chi_{\{|x|\leq N\}}))^2\otimes e_k^{\ell_{p,q}} \right\|_{L_{\frac{p}{2},\frac{q}{2}}{(\cM\overline{\otimes}\ell_\infty )}}.$$
Since the quasi-norms in both spaces $L_{p,q}(\mathcal{M},\tau)$ and $L_{p,q}(\mathcal{M}\overline{\otimes} \ell_{\infty},\tau\otimes\Sigma)$ are order continuous (see e.g. \cite[Proposition 3.1]{DDS2014}), it follows that
$$\left\|
\sum_{k=0}^n\left(i_k(x)\right)^2\right\|_{L_{\frac{p}{2},\frac{q}{2}}(\cM,\tau)} \leq c_{p,q}\left \|\sum_{k=0}^n(i_k(x))^2\otimes e_k^{\ell_{p,q}} \right\|_{L_{\frac{p}{2},\frac{q}{2}}{(\cM\overline{\otimes}\ell_\infty )}}.$$
Therefore, we have
\begin{align*}
\left\|B_nx\right\|_{L_1(L_{p,q}(\mathcal{M},\tau)) }& \lesssim_{p,q} \left\|\sum_{k=0}^n\left( i_k(x)\right)^2\otimes e_k^{\ell_{p,q}} \right\|_{L_{\frac{p}{2},\frac{q}{2}}{(\cM\overline{\otimes}\ell_\infty )}} ^{\frac12} \\
&=~\left\|(\mu(x)^2)^{\oplus(n+1)}\right\|_{L_{\frac{p}{2},\frac{q}{2}}(0,\infty )}^{\frac12}=(n+1)^{\frac1p}\left\|x\right\|_{L_{p,q} (0,1)}.
\end{align*}
This completes the proof. 
\end{proof}
The following estimate is an analogy of that in Lemma \ref{anx bound noncomm} above. 
\begin{lem}\label{bnx norm} Let $p<2.$ Let the operators $B_n$ be as in Lemma \ref{bn norm}. For every $x\in L_{p,q}(0,1),$ we have $$\left\|B_nx\right\|_{L_1(L_{p,q})}=o(n^{\frac1p}),$$
as  $n\to\infty.$
\end{lem}
\begin{proof} By Theorem \ref{vector lust-piquard}, for any $z\in L_{\infty }(0,1)$, we have
$$\left\|B_nz\right\|_{L_1(L_{p,q}(\cM,\tau))}\lesssim\left\|
(\sum_{k=0}^n(i_k(z))^2)^{\frac12}\right\|_{L_{p,q}(\cM,\tau)}
=\left\|\sum_{k=0}^n(i_k(z))^2\right\|_{ {L_{\frac{p}{2},\frac{q}{2}}(\cM,\tau)} }^{\frac12}.$$
By the triangular inequality, we have 
$$\left\|B_nz\right \|_{L_1(L_{p,q}(\cM,\tau) )}\lesssim 
\left( \sum_{k=0}^n\left\|(i_k(z))^2\right\|_{ {L_{\frac{p}{2},\frac{q}{2}}(\cM,\tau)} } \right) ^{\frac12}\le \left( \sum_{k=0}^n \left\|z\right\|_{\infty}^ 2 \right) ^{\frac12}\le 
(n+1)^{\frac12} \left\|z\right\|_{\infty}.$$
Since $L_{p,q}(0,1)$ has order continuous norm, it follows that for any 
fixed $\epsilon>0$, there exists   $y\in L_{\infty}(0,1)$ such that $\left\|x-y\right\|_{p,q}\leq\epsilon.$ It follows that
\begin{align*}
\left\|B_nx\right\|_{L_1(L_{p,q}(\cM,\tau))}&~\quad \leq~ \quad  \left\|B_ny\right\|_{L_1(L_{p,q}(\cM,\tau))}+\left\|B_n(x-y)\right\|_{L_1(L_{p,q}(\cM,\tau))}\\
&\stackrel{\tiny \mbox{Lemma \ref{bn norm}}}{\lesssim}  n^{\frac12}\left\|y\right\|_{\infty}+n^{\frac1p}\epsilon.
\end{align*}
Taking into account that $p<2$, we have 
$$\limsup_{n\to\infty}n^{-\frac1p}\left\|B_nx\right\|_{L_1(L_{p,q}(\cM,\tau))}\lesssim\epsilon.$$
Since $\epsilon>0$ is arbitrarily taken, the assertion follows.
\end{proof}

\section{Proof of Theorem \ref{main theorem}}\label{noncommutative section}
The  goal of this section is to prove Theorem \ref{main theorem}. 

Assume that $\ell_F$ is a separable symmetric sequence space and $E(0,1)$ is a KB-symmetric function space (i.e. a symmetric function space having the Fatou property and order continuous norm). 
Let $T:\ell_{F}\to E(\mathcal{M},\tau)$ be an isomorphic  embedding.

Set $(\mathcal{M}_0,\tau_0)=(\mathcal{M}\oplus\mathcal{M},\frac12\tau\oplus\frac12\tau).$ Set 
$$T_0(\alpha)=\Re(T(\alpha))\oplus\Im(T(\alpha)),\quad \alpha\in \ell_{F}^{\mathbb{R}},$$
which  is also an isomorphic embedding
from $\ell_{F}$ into $E(\cM_0,\tau_0)$.
Since  $T_0\alpha$ is an self-adjoint operator in $E(\cM_0,\tau_0)$ for every $\alpha\in \ell_{F}^{\mathbb{R}}$, one may 
 assume without loss of generality that $T\alpha$ is self-adjoint for every $\alpha\in \ell_{F}^{\mathbb{R}}.$
 
By replacing $\cM$ with $\cM\overline{\otimes} L_\infty(0,1)$, we may assume that $\cM$ is atomless. 
Denote by  $\{e_k\}_{k\geq0}$ be the standard unit basis in $\ell_{F}$ and denote by  $x_k=Te_k\in E (\mathcal{M},\tau),$ $k\geq0.$ 
By the preceding paragraph, we may assume without loss of generality that $x_k=x_k^*.$ 

For any $\alpha=(\alpha(k))_{k\ge 0}\in \ell_{F}$ and for   every $t\in(0,1),$ we have
\begin{align}\label{all t equivalence}
\left\|\sum_{k\geq0}\alpha(k)r_k(t)x_k\right\|_{E(\cM,\tau)}&\quad =\quad \left\|
T\left(\sum_{k\geq0}\alpha(k)r_k(t)e_k\right)\right\|_{E(\cM,\tau)}\nonumber \\
&\quad \approx_{C(T)}\left\|
\sum_{k\geq0}\alpha(k)r_k(t)e_k\right\|_{\ell_{F}}= \left\|\alpha\right\|_{\ell_{F}}.
\end{align}
 
By Theorem \ref{subsequence splitting property}, 
passing to a subsequence if needed, we may  write
$$x_k=u_k+v_k+w_k,$$
where  $u_k,v_k,w_k\in E (\mathcal{M},\tau),$ $k\geq0,$ is self-adjoint, 
  $ u:=\mu(u_1)=\mu(u_k)$ for all $k\geq0$,
elements of the sequence $\{v_k\}_{k\geq0}$ are mutually orthogonal, i.e. $v_kv_l=0,$ $k\neq l$,
elements $w_k\to0$ in $E (\mathcal{M},\tau)$ as $k\to\infty.$
Let $i_k:S(0,\infty)\to S(\cM,\tau)$ be a  $*$-homomorphism preserving singular value functions such that $i_k(u)=u_k$.

Set $y_k=u_k+v_k,$ $k\geq0.$
Passing to a subsequence if needed, we may assume that (see e.g. Proposition 1.a.3 in \cite{LT2})
$$\{x_k\}_{k\ge 0}\sim \{y_k\}$$
In particular, by \eqref{all t equivalence},  we have
\begin{equation}\label{vector-valued equivalence}
\int_0^1\left \|\sum_{k\geq0}\alpha_kr_k(t)y_k\right \|_{E (\cM,\tau)}dt\approx_{C(T)}\left\|\alpha\right \|_{\ell_{F}}
\end{equation}
and
\begin{equation}\label{uniform vector-valued equivalence}
\inf_{t\in(0,1)}\left\|\sum_{k\geq0}\alpha_kr_k(t)y_k\right \|_{E(\cM,\tau)}\approx \sup_{t\in(0,1)}\left\|\sum_{k\geq0}\alpha_kr_k(t)y_k\right \|_{ E (\cM,\tau)}\approx_{C(T)}\left\|\alpha\right\|_{\ell_{F}}.
\end{equation}

Assume, by way of  contradiction, that $\ell_{p,q}$ is isomorphic to a subspace of $L_{p,q}(\mathcal{M},\tau),$ $ (p,q)\in (1,\infty)\times [1,\infty )$, $p\neq q$. 
Let $\ell_F=\ell_{p,q}$ and $E(0,1)=L_{p,q}(0,1)$. 
There are two possibilities  for $\{v_k\}_{k\ge  0}$:
\begin{enumerate}
  \item if there exists a subsequence of $\{v_k\}_{k\ge 0}$ converging to $0$ in norm, then, passing to a subsequence if necessary  \cite[Proposition 1.a.3]{LT2}, we may assume that 
      \begin{align}\label{case1}
      \{u_k\}_{k\ge 0}\sim \{y_k\}_{k\ge 0}\sim \{x_k\}_{k\ge 0};
      \end{align}
  \item if there exists a constant $\delta>0$ such that $\norm{v_k}_{p,q}>\delta$ for all $k\ge 0$, then,  passing to a subsequence if needed, we infer from Lemma \ref{2.1} that 
      \begin{align}\label{case2}
      \{v_k\}_{k\geq0}\sim \{e^{\ell_q}_k\}_{k\geq0}.
       \end{align} 
\end{enumerate}

Below, we prove Theorem \ref{main theorem} case by case.  
 
\subsection{The case when $p> 2$}
In this subsection,  we prove a much more general result  for intermediate spaces  for the Banach couple  $L_2$ and $L_\infty$ (see \cite[Chapter I.3]{KPS}).
\begin{theorem}\label{sequence_com_nc}
Let $2<  r <\infty$ and let 
  $\ell_F$ be a separable symmetric sequence space such that $\ell_r \subset \ell_F \subset \ell_\infty$.
If  $E(0,1) $ is  a KB-symmetric function space such that  $L_\infty (0,1) \subset E(0,1)\subset L_r(0,1) $,    
then
 $$\ell_F\hookrightarrow E(0,1) \mbox{ if and only if }   \ell_F  \hookrightarrow E(\cM,\tau).$$
\end{theorem}
\begin{proof}

It suffices to show that   $ \ell_F  \hookrightarrow E(\cM,\tau)$ implies that $\ell_F\hookrightarrow E(0,1) $.

Let $T:\ell_F \to E(\cM,\tau)$ be an isomorphic embedding.
Since $E(0,1)$ is an intermediate space between $L_r  (0,1) $ and $L_\infty(0,1)  $,
 it follows from  \cite[Chapter I, Theorem 1.8]{BS}
 that $$\left\|\cdot \right\|_{L_2(\cM,\tau)} \lesssim  \norm{\cdot}_{L_{r}(\cM,\tau)} \lesssim  \left\|\cdot\right\|_{E(\cM,\tau)}\lesssim  \left\|\cdot\right\|_{L_\infty (\cM,\tau)} , $$
 and since $\ell_r\subset \ell_F\subset \ell_\infty $, it follows that 
 $$\norm{\cdot}_{\ell_\infty }\lesssim  \left\|\cdot\right\|_{\ell_F} \lesssim \norm{\cdot}_{\ell_r}\lesssim\norm{\cdot}_{\ell_2}.$$
  It follows from the triangle inequality of $L_r(\cM,\tau)$ that, passing to a subsequence if necessary, we have 
\begin{align}\label{ineqr} &~\qquad\qquad 
\int_0^1\left\|
\sum_{k\geq0}\alpha_kr_k(t)u_k\right\|_{L_r(\cM,\tau)}dt\nonumber\\
 &~\qquad \leq \quad  \int_0^1\left\|
\sum_{k\geq0}\alpha_kr_k(t)y_k\right\|_{L_r(\cM,\tau)} dt+\int_0^1\left\|
\sum_{k\geq0}\alpha_kr_k(t)v_k
\right\|_{L_r (\cM,\tau)} dt\nonumber\\
 &~\qquad \le \quad  \int_0^1\left\|
\sum_{k\geq0}\alpha_kr_k(t)y_k\right\|_{E(\cM,\tau)} dt+\int_0^1\left\|
\sum_{k\geq0}\alpha_kr_k(t)v_k
\right\|_{L_r (\cM,\tau)} dt\nonumber\\
&\qquad \stackrel{\eqref{vector-valued equivalence} }{\lesssim }  \quad  \left\|\alpha\right\|_{\ell_F}  +\norm{\alpha}_{\ell_r }  \lesssim \norm{\alpha}_{\ell_r}  .
\end{align}
  It follows from Kahane inequality (see e.g. \cite[Theorem 6.2.5]{AK}) that
\begin{align*}
 \left\|\alpha\right\|_{\ell_2} \left\|u\right\|_{L_2(\cM,\tau)}&=
\left(\int_0^1 \left\|\sum_{k\geq0}\alpha_kr_k(t)u_k\right\|_{L_2(\cM,\tau)}^2dt\right)^{\frac12}
\\
&\lesssim\int_0^1 \left\|\sum_{k\geq0}\alpha_kr_k(t)u_k\right\|_{L_2(\cM,\tau)} dt\\
&\lesssim\int_0^1 \left\|\sum_{k\geq0}\alpha_kr_k(t)u_k \right \|_{L_r (\cM,\tau)} dt\stackrel{\eqref{ineqr}}{\lesssim}  \left\|
\alpha \right \|_{\ell_r}  .
\end{align*}
If $u\neq0,$ then $\left\|\alpha\right\|_{\ell_r }\gtrsim \left\|\alpha\right\|_{\ell_2}$ for every $\alpha\in \ell_{2} $, which is a contradiction with the fact  that  $   \norm{\cdot}_{\ell_2} \ge \norm{\cdot}_{\ell_r}  $.
Therefore, $u_k=0$ for all $k\ge 0$.

Passing to a subsequence if necessary, $\{T(e_n^{\ell_F})\}_{n\ge 0}$ in $\ell_F$ is equivalent to a sequence  $\{v_k\}_{k\ge 0}$   of 
pairwise orthogonal self-adjoint elements  in $E(\cM,\tau)$.
Clearly, $\{v_k\}_{k\ge 0}$ in $E(\cM,\tau)$ is equivalent to a sequence disjointly supported elements in $E(0,1)$.
This implies that $\ell_{F}$ embeds into $E(0,1)$.
\end{proof}
Note that $\ell_{\frac{2+p}{2}}\subset \ell_{p} $ and $L_{p,q}(0,1)\subset L_{\frac{2+p}{2}} $ for any $p>2$. The following result is a straightforward consequence of Theorem \ref{sequence_com_nc} and the result $\ell_{p,q}\not\hookrightarrow L_{p,q}(0,1)$ obtained in \cite{KS}. 
\begin{cor}
  If $p>2$, $1\le q<\infty$ and $p\ne q$, then 
  $$\ell_{p,q}\not\hookrightarrow L_{p,q}(\cM,\tau). $$
\end{cor}

\subsection{Proof of Theorem \ref{main theorem} for  case  $p<2$}
\begin{proof} We may assume that  $\{v_k\}_{k\geq0}$ is  equivalent to a unit vector basis in $\ell_q$ or that $v_k=0,$ $k\geq0.$

By the triangle inequality of $L_{p,q}(\cM,\tau)$, we have
\begin{align*}&~\quad \left|\int_0^1\left\|\sum_{k=0}^nr_k(t)y_k\right\|_{L_{p,q}(\cM,\tau)}dt
-\int_0^1\left\|\sum_{k=0}^nr_k(t)v_k\right\|_{L_{p,q}(\cM,\tau)}dt\right|\\
&\leq\int_0^1\left \|\sum_{k=0}^nr_k(t)u_k\right\|_{L_{p,q}(\cM,\tau)} dt.
\end{align*}
By Lemma \ref{bnx norm} and the definition of $B_n$, we have
$$\int_0^1\left\|\sum_{k=0}^nr_k(t)u_k\right\|_{L_{p,q}(\cM,\tau)}dt=\int_0^1\left\|(B_n u)(t) \right\|_{L_{p,q}(\cM,\tau)}dt=\left\|B_nu\right\|_{L_1(L_{p,q}(\cM,\tau))}=o(n^{\frac1p}),$$
as $n\to\infty.$
By  \eqref{vector-valued equivalence}, we have 
$$\int_0^1\left\|\sum_{k=0}^nr_k(t)y_k\right\|_{L_{p,q}(\cM,\tau)}dt\approx \norm{\sum_{k=0}^n e^{\ell_{p,q}}_k}_{\ell_{p,q}}\approx n^{\frac1p}.  $$
Combining previous estimates, we obtain (for sufficiently large $n$)
$$\int_0^1\left\|\sum_{k=0}^nr_k(t)v_k\right\|_{L_{p,q}(\cM,\tau)}dt\approx n^{\frac1p}.$$
If $\{v_k\}_{k\geq0}=0,$ then we obtain a contradiction. If $\{v_k\}_{k\geq0}$ is equivalent to a unit vector basis in $\ell_q,$ then
$$n^{\frac1q}\approx\left\|\sum_{k=0}^nv_k\right\|_{L_{p,q}(\cM,\tau)}=\int_0^1\left\|\sum_{k=0}^nr_k(t)v_k\right\|_{L_{p,q}(\cM,\tau)}dt\approx n^{\frac1p}.$$
 This is impossible for $p\neq q$, which    completes the proof.
\end{proof}

\subsection{Proof of Theorem \ref{main theorem} for  case   $p=2,$ $1\leq q<2$}

\begin{proof}{\bf Step 1:} Suppose that
\eqref{case2} holds, i.e.,  $\left\{v_k\right\}_{k\geq0}$ is equivalent to the  unit vector basis in $\ell_q.$

It follows from the triangle inequality of $L_{2,q}(\cM,\tau)$  that for any $n\ge 0$, we have 
\begin{align*}
 \int_0^1\left\| \sum_{k=0}^n(k+1)^{-\frac12}r_k(t)v_k\right \|_{L_{2,q}(\cM,\tau)}dt &\leq\int_0^1\left \|\sum_{k=0}^n(k+1)^{-\frac12}r_k(t)y_k\right\|_{L_{2,q}(\cM,\tau)}dt\nonumber \\
&~\quad  +\int_0^1\left\|
\sum_{k=0}^n(k+1)^{-\frac12}r_k(t)u_k\right\|_{L_{2,q}(\cM,\tau)}dt.
\end{align*}
 We have 
\begin{align*}\int_0^1\left \|\sum_{k=0}^n(k+1)^{-\frac12}r_k(t)y_k\right\|_{2,q}dt&\stackrel{\eqref{vector-valued equivalence}}{\leq} C(T)\left \|\left\{(k+1)^{-\frac12}\right\}_{k=0}^n\right\|_{\ell_{2,q} }\\
&\stackrel{\eqref{lpqsequence}}{=}C(T) \left(
\sum_{k=0}^n(k+1)^{-\frac{q}{2}}\cdot(k+1)^{\frac{q}{2}-1}\right)^{\frac1q}\\
&~=~C(T) \left(
\sum_{k=0}^n(k+1)^{-1}\right)^{\frac1q}=O(\log^{\frac1q}(n)).
\end{align*}
By Theorem \ref{vector vs tensor}, we have
\begin{align*}
\int_0^1\left\|\sum_{k=0}^n(k+1)^{-\frac12}r_k(t)u_k\right\|_{L_{2,q}( \cM )}dt &~\leq c_{p,q}\left\|\sum_{k=0}^n(k+1)^{-\frac12}r_k\otimes u_k\right\|_{L_{2,q}(L_\infty \overline{\otimes} \cM )}\\
&~=c_{p,q}\left\|\sum_{k=0}^n(k+1)^{-\frac12}r_k\otimes i_k(u)\right\|_{L_{2,q}(L_\infty \overline{\otimes} \cM )}\\
&\stackrel{\eqref{estiAn}}{=}c_{p,q}\left\|A_n(u)\right\|_{L_{2,q}(L_\infty \overline{\otimes} \cM )}\\
&~\leq ~ c_{p,q}\left\|A_n\right \|_{L_{2,q}(0,1)\to _{L_{2,q}(L_\infty \overline{\otimes} \cM )}}\cdot \left\|u\right\|_{L_{2,q}(0,1)}\\
& \stackrel{L.\ref{an l2q noncomm}}{=}  c_{abs}\log^{\frac1q}(en) \cdot \norm{u}_{L_{2,q}(0,1)}.
\end{align*}
It follows that
$$\int_0^1\Big\|\sum_{k=0}^n(k+1)^{-\frac12}r_k(t)v_k\Big\|_{2,q}dt=O(\log^{\frac1q}(n)).$$
However, since $\{v_k\}_{k\geq0}$ is equivalent to the  unit vector basis in $\ell_q$, it follows that  
\begin{align*}
\int_0^1\left\|\sum_{k=0}^n(k+1)^{-\frac12}r_k(t)v_k\right\|_{L_{2,q}(\cM,\tau)}dt&=
\left\|\sum_{k=0}^n(k+1)^{-\frac12}v_k\right\|_{L_{2,q}(\cM,\tau)}\\
&\approx\left\|\Big\{(k+1)^{-\frac12}\Big\}_{k=0}^n\right\|_{\ell_q}.
\end{align*}
Thus,
$$n^{\frac1q-\frac12}\approx\Big\|\Big\{(k+1)^{-\frac12}\Big\}_{k=0}^n\Big\|_q= O(\log^{\frac1q}(n)),$$
which is impossible. 
This contradiction shows that $v_k =0$ for all $k\geq0$.

{\bf Step 2:} Now, suppose that \eqref{case1} holds, i.e.,  $\{v_k\}_{k\geq0}=0;$ thus, $y_k=u_k.$ It follows from \eqref{vector-valued equivalence} that
\begin{align}\label{second vector-valued equivalence}
\int_0^1\left\|\sum_{k=0}^n(k+1)^{-\frac12}r_k(t)u_k\right\|_{L_{2,q}(\cM,\tau)}dt&\approx \norm{\sum_{k=0}^n(k+1)^{-\frac12} e^{\ell_{2,q}}_k  }_{\ell_{2,q}}\nonumber  \\
&\approx
\left(
\sum_{k=0}^n(k+1)^{-\frac{q}{2}}\cdot(k+1)^{\frac{q}{2}-1}\right)^{\frac1q} \\
&\approx 
\log^{\frac1q}(n).\nonumber
\end{align}
However, using Theorem \ref{vector vs tensor} again, we obtain that
\begin{align*}
\int_0^1\left\| \sum_{k=0}^n(k+1)^{-\frac12}r_k(t)u_k\right\|_{L_{2,q}(\cM,\tau)}dt&~\leq~ c_{abs}\left\|\sum_{k=0}^n(k+1)^{-\frac12}r_k\otimes u_k\right\|_{L_{2,q}(L_\infty \overline{\otimes} \cM )}\\
&~=~\left\| \sum_{k=0}^n(k+1)^{-\frac12}r_k\otimes i_k(u)\right\|_{L_{2,q}(L_\infty \overline{\otimes} \cM )}\\
&~=~\left\|A_n(u)\right\|_{L_{2,q}(L_\infty \overline{\otimes} \cM )}\\
&\stackrel{L.\ref{anx bound noncomm}}{=}o(\log^{\frac1q}(n)).
\end{align*}
This contradicts \eqref{second vector-valued equivalence} and we complete the proof for the case when  $p=2,$ $1\leq q<2$.
\end{proof}

\subsection{Proof of Theorem \ref{main theorem} for  case   $p=2,$ $q>2$}
\begin{proof} We either have that $\{v_k\}_{k\geq0}$ is equivalent to a unit vector basis in $\ell_q$ or that $v_k=0,$ $k\geq0.$

It follows from the triangle inequality of $L_{2,q}(\cM,\tau)$  that
\begin{align*}
&~\quad \sup_{t\in(0,1)}\left\|\sum_{k\geq0}\alpha_kr_k(t)u_k\right \|_{L_{2,q}(\cM,\tau)}\\
&\leq\sup_{t\in(0,1)}\left\|\sum_{k\geq0}\alpha_kr_k(t)y_k\right\|_{L_{2,q}(\cM,\tau)}+\sup_{t\in(0,1)}\left\|\sum_{k\geq0}\alpha_kr_k(t)v_k\right\|_{L_{2,q}(\cM,\tau)}.
\end{align*}
By \eqref{uniform vector-valued equivalence}, we have $\sup_{t\in(0,1)}\left\|\sum_{k\geq0}\alpha_kr_k(t)y_k\right\|_{L_{2,q}(\cM,\tau)}\approx \left\|\alpha\right \|_{\ell_{2,q} }.$ 
Passing to a subsequence of $\{v_k\}_{k\ge 0}$ if necessary, 
$$\sup_{t\in(0,1)}\left\|\sum_{k\geq0}\alpha_kr_k(t)v_k\right\|_{L_{2,q}(\cM,\tau)} =\left\| \sum_{k\geq0}\alpha_kv_k\right\|_{L_{2,q}(\cM,\tau)}\lesssim \left\|\alpha\right\|_{\ell_q}.$$
Since $\norm{\cdot}_{\ell_q}\le \norm{\cdot}_{\ell_{2,q}}$\cite[p.217]{BS}, it follows that 
$$\sup_{t\in(0,1)}\left\|\sum_{k\geq0}\alpha_kr_k(t)u_k
\right\|_{L_{2,q}(\cM,\tau)} \lesssim \left\|\alpha\right\|_{\ell_{2,q}}+\left\|\alpha\right\|_q\lesssim \left\|\alpha\right\|_{\ell_{2,q}}.$$
By Theorem \ref{vector vs tensor} (b), we have
\begin{align}\label{lastpara}
\left\|\sum_{k\geq0}\alpha_kr_k\otimes u_k\right\|_{L_{2,q}(L_\infty(0,1)\overline{\otimes} \cM)}\lesssim \left\|\alpha\right\|_{\ell_{2,q}}.
\end{align}

Recall that  $u_k=i_k(u),$ $k\geq0.$ It follows from the preceding paragraph that
\begin{align*}
\left\|A_n(u)\right\|_{L_{2,q}(L_\infty \overline{\otimes} \cM )}&~=~\left\| \sum_{k=0}^n(k+1)^{-\frac12}r_k\otimes i_k(u)\right\|_{L_{2,q}(L_\infty \overline{\otimes} \cM )}\\
&~=~\left\| \sum_{k=0}^n(k+1)^{-\frac12}r_k\otimes  u_k  \right\|_{L_{2,q}(L_\infty \overline{\otimes} \cM )}\\
&\stackrel{\eqref{lastpara}}{\lesssim} \left\|\Big\{(k+1)^{-\frac12}\Big\}_{k=0}^n\right\|_{\ell_{2,q} }\stackrel{\eqref{second vector-valued equivalence}}{=}O(\log^{\frac1q}(n)).
\end{align*}
Recall that $L_{2,q'}$ is the Banach dual of $L_{2,q},$ $\frac1q+\frac1{q'}=1$ (with equivalent norms)\cite[Corollary 4.8]{BS}. We obtain that 
\begin{align}\label{sgn1}
(m\otimes \tau)\left(A_n(u)\cdot A_n({\rm sgn}(u)) \right) &\leq c_q \left\|A_n(u)\right\|_{L_{2,q}(L_\infty \overline{\otimes} \cM )} \left\|A_n({\rm sgn}(u))\right\|_{L_{2,q'}(L_\infty \overline{\otimes} \cM )}\nonumber \\
&\stackrel{L. \ref{anx bound noncomm}}{=}O(\log^{\frac1q}(n))\cdot o(\log^{\frac1{q'}}(n))=o(\log(n)).
\end{align}
On the other hand,
by the definition of $A_n$ (see \eqref{estiAn}), 
 we have
\begin{align*}
A_n(u)\cdot A_n({\rm sgn}(u))
&=
 \left(\sum_{k=0}^n(k+1)^{-\frac12}r_k\otimes i_k(u) \right) \cdot  \left(\sum_{k=0}^n(k+1)^{-\frac12}r_k\otimes i_k({\rm sgn}(u)) \right)\\&
=\sum_{k,l=0}^n(k+1)^{-\frac12}(l+1)^{-\frac12}r_kr_l\otimes i_k(u)i_l({\rm sgn}(u) ).
\end{align*}
Therefore,
$$(m\otimes \tau)\left(A_n(u)\cdot A_n({\rm sgn}(u)) \right) =\left(\sum_{k=0}^n\frac1{k+1}\right)\left\|u\right\|_{L_1(0,1)}  =O(\log(n)),$$
which contradicts \eqref{sgn1} whenever $u\ne 0$. Hence, $u=0$.

Since $u=0$, it follows that  $u_k=0$ for $k\geq0;$ thus, $y_k=v_k,$ $k\geq0.$ In particular, 
$$\int_0^1\left\|\sum_{k\geq0}\alpha_kr_k(t)y_k\right\|_{L_{2,q}(\cM,\tau)}dt=\int_0^1 \left\|\sum_{k\geq0}\alpha_kr_k(t)v_k\right\|_{L_{2,q}(\cM,\tau)}dt.$$
It follows  from \eqref{case2} that  
$$\int_0^1 \left\|\sum_{k\geq0}\alpha_kr_k(t)v_k\right\|_{L_{2,q}(\cM,\tau)}dt=\left\|
\sum_{k\geq0}\alpha_kv_k\right\|_{L_{2,q}(\cM,\tau)} \approx  \left\|\alpha\right\|_{\ell_q}.$$
It follows from \eqref{vector-valued equivalence} that
$$\left\|\alpha\right\|_{\ell_{2,q}}\approx\int_0^1 \left\|\sum_{k\geq0}\alpha_kr_k(t)y_k\right\|_{L_{2,q}(\cM,\tau)}dt=\int_0^1 \left\|\sum_{k\geq0}\alpha_kr_k(t)v_k\right\|_{L_{2,q}(\cM,\tau)}dt \approx \left\|\alpha\right\|_{\ell_q},$$
which is impossible\cite[p.217]{BS}. 
This completes the proof. 
\end{proof}

\section{Applications: towards the isomorphic classification of noncommutative spaces $L_{p,q}$ on von Neumann algebras of type $I$}\label{application}
In this section, we present the proof of Hasse diagram given in the Introduction.

If $\cM=B(\cH)$ and $\tau$ is the standard trace, then we denote $L_{p,q}(\cM,\tau)$ by $C_{p,q}$.
If $\cM=\oplus_{1\le n<\infty}\mathbb{M}_n$ equipped with the trace $\oplus_{1\le n<\infty}{\rm Tr}_n$, where ${\rm Tr}_n$ is the standard trace on the space of $n\times n$ matrices, then we denote 
$$S_{p,q}:= L_{p,q}(\cM,\tau).$$
It is well-known that 
$$S_{p,q}\approx L_{p,q}(\cM,\tau)$$
for any $\cM=\oplus_{1\le n<\infty}\mathbb{M}_n\oplus \cA_n$, where $\cA_n$ is an atomic commutative  von Neumann algebra with all atoms in $\cM$ having the same trace,  and for any $N>0$, there exists an $n>N$ such that $\cA_n\ne \{0\}$ \cite[p.335]{Arazy81}.

By \cite[Lemma 2.5]{AHS}, we have  $C_{p,q}\approx  T_{p,q}:=\{x\in C_{p,q}: x_{ij}=0,i>j\}$, $1<p<\infty$, $1\le q<\infty $. 
The lack of isomorphic embeddings of $L_{p,q}(0,1)$ into $C_{p,q}$ was discussed in \cite{AHS}. However, the case when $p>2$, $q\in [1,2]\cup (p,\infty)$ remains open. Below, we show the lack of isomorphic embedding of $L_{p,q}(0,\infty)$ into $C_{p,q}$.  
 
\begin{prop}\label{LpqintoCpq}
Let $p>2$. Then, $$U_{p,q}\not\hookrightarrow C_{p,q}.$$
Consequently, $L_{p,q}(0,\infty)\not\hookrightarrow C_{p,q}$ for all $(p,q)\in (1,\infty)\times [1,\infty)$, $(p,q)\ne (2,2)$. 
\end{prop}
\begin{proof}
Assume 
 that there exists an isomorphism $T$ from $U_{p,q}$ into $T_{p,q}=\{x\in C_{p,q}: x_{ij}=0,i>j\}$,
which is isomorphic to  $ C_{p,q}$. 
Without loss of generality, we may assume that 
$\{A_n\}=\cup_n B_n$ and 
$U_{p,q}=L_{p,q}\{ \{ A_n\}_{n\ge 1}\}$, where for each $n\ge 1$,     $B_n=\{A_k^n\}_{k=1}^\infty$ is    a set of infinite many atoms of measure $\frac1n$.
Define $\{e_{j}^n\}$ be the normalized basis of the characteristic functions $\chi_{_{A_j^n}}$'s.
We have  $e^n_j = \frac{1}{\norm{\chi_{_{A_j^n}}}_{U_{p,q}}}\chi_{_{A_j^n}}$ and  $\norm{\chi_{_{A_j^n}}}_{U_{p,q}}=\norm{\chi_{(0,\frac1n)}}_{L_{p,q}(0,\infty)}=\frac{1}{n^{1/p}}$.

Let $\{e_{ij}\}_{i,j\ge 1}$ be the unit matrix basis of $C_{p,q}$. 
Denote $E_n=\sum_{i=1}^n e_{ii}=\begin{pmatrix}1 & 0 &0&  \cdots\\
0&1 & 0&\cdots\\
0&0 & 1 &\cdots\\
\vdots& \vdots&\vdots&\ddots\\
\end{pmatrix}={\rm diag }(\underbrace{1,\cdots,1}_{n~terms},0,0,\cdots)$ .

We claim that 
for any fixed $k>0$\footnote{$E_k x$ is the restriction of $x$ on the first $k$ rows, and 
$ xE_k$ is the restriction of $x$ on the first $k$ columns.} and any fixed $n\ge 1$, 
\begin{align}\label{disappear}
\norm{E_{k} T(e_i^n) }_{C_{p,q}}\to 0\mbox{ as }i\to \infty.
\end{align}
Indeed, assume, by way of  contradiction, that (passing to a subsequence if necessary) that
there exists $n\ge 1$ such that 
 $\norm{E_{k} T(e_i^n) }_{C_{p,q}}>\delta >0$ for all $j\ge 2$.

Let $(\varepsilon_1,\varepsilon_2,\cdots)$ a decreasing sequence of positive numbers. 
Since $C_{p,q}$ is separable, it follows that there exists $m_1$ such that 
$$\norm{E_k T(e_1^n)E_{m_1} -E_k T(e_1^n)}_{C_{p,q}}< \varepsilon_1\mbox{ and }\norm{E_k T(e_1^n)E_{m_1} }_{C_{p,q}}> \frac{\delta}{2}.$$
Since $e_j^n\to 0$ weakly as $j\to \infty $ and $E_k C_{p,q}E_{m_1}$ is finite-dimensional, it follows that there exists $j_2$ such that 
$$\norm{E_k  T(e_{j_2}^n) E_{m_1}}_{C_{p,q}}  <\frac{\varepsilon_2}{2}.$$
Again, since $C_{p,q}$ is separable, it follows that there exists $m_2$ such that 
\begin{align*}&~\quad \norm{E_k T(e_{j_2}^n)(E_{m_2} -E_{m_1}) - E_k  T(e_{j_2}^n)}_{C_{p,q}}\\
&\le 
\norm{E_k T(e_{j_2}^n)E_{m_1}}_{C_{p,q}}+\norm{E_k T(e_{j_2}^n) E_{m_2} -E_k  T(e_{j_2}^n)}_{C_{p,q}}< \varepsilon_2
\end{align*}
and 
$$\norm{E_k T(e_{j_2}^n)(E_{m_2}-E_{m_1}) }_{C_{p,q}}>\frac\delta2.$$
Arguing inductively, we obtain an increasing sequence $\{j_i\}_{i\ge 0}$ (denote $j_0=0$ and $j_1=1$)  such that 
\begin{align}\label{epsiloni}
&~\quad \norm{E_k T(e_{j_2}^n)(E_{m_i} -E_{m_{i-1}})- E_kT(e_{j_2}^n)}_{C_{p,q}}< \varepsilon_i
\end{align}
and 
\begin{align}\label{EKT}
\norm{E_k T(e_{j_i}^n)(E_{m_i} -E_{m_{i-1}})  }_{C_{p,q}}>\frac\delta2.
\end{align}
To lighten the notions, we may replace $\left\{e_{j_i}^n\right\}_{i\ge 1}$ with $\left\{e_i^n\right\}_{i\ge 1}$. 
Passing to a subsequence if necessary, there exists $1\le j\le k$ such that 
\begin{align}\label{EJEJ}
 \frac{\delta}{2k} \stackrel{\eqref{EKT}}{<}
\norm{(E_j-E_{j-1}) T(e_{i}^n)(E_{m_i} -E_{m_{i-1}})  }_{C_{p,q}}\le \norm{T}.
\end{align}
Noting that $(E_j-E_{j-1})C_{p,q}$ is isometric to $\ell_2$, we obtain that for any $a=(a_1,a_2,\cdots)$,  
\begin{align*}
\frac{\delta}{2k}\norm{a}_{\ell_2}
&~\le  \norm{(E_j-E_{j-1}) \sum_{i\ge 1} a_i T(e_{i}^n)(E_{m_i} -E_{m_{i-1}})   }_{C_{p,q}}\\
 &~\le  \norm{E_k\sum_{i\ge 1} a_i T(e_{i}^n)(E_{m_i} -E_{m_{i-1}}) }_{C_{p,q}}\\
 &~\le \sum_{j=1}^k  \norm{ (E_j-E_{j-1}) \sum_{i\ge 1} a_i T(e_{i}^n)(E_{m_i} -E_{m_{i-1}}) }_{C_{p,q}}\\
 &~=  \sum_{j=1}^k \left(
\sum_{i\ge 1} \norm{ (E_j-E_{j-1})  a_i T(e_{i}^n)(E_{m_i} -E_{m_{i-1}}) }_{C_{p,q}}^2\right)^{1/2}\\
&\stackrel{\eqref{EJEJ}}{=}  \sum_{p=1}^k \left(
\sum_{i\ge 1} a_i^2  \norm{   T  }^2\right)^{1/2}\\
&~\le k  \norm{T}\norm{a}_{\ell_2} .
\end{align*}
Therefore, by \cite[Theorem 1.3.9]{AK}, for suitable choice of $(\varepsilon_1,\varepsilon_2,\cdots)$, we obtain that
$$\{E_k T(e_{i}^n) \}_{i\ge 1}\stackrel{\eqref{epsiloni}}{\sim}\{E_k T(e_{i}^n)(E_{m_i} -E_{m_{i-1}})\}_{i\ge 1}\sim \{e^{\ell_2}_i\}_{i\ge 1}.$$
 For any  sequence $\alpha:=(\alpha_1,\cdots, \alpha_m,0,\cdots)$, we have 
$$ \frac{1}{n^{1/p}} \norm{\alpha}_{\ell_{p,q}}= \norm{\sum_{i=1}^m \alpha_i e_i^n }_{U_{p,q}}
\sim \norm{T(\sum_{i=1}^m \alpha_i e_{i}^n ) }_{C_{p,q}}
\ge
 \norm{E_{k} T(\sum_{i=1}^m \alpha_i e_{i}^n ) }_{C_{p,q}}
\sim 
\norm{\alpha}_{\ell_2},$$
which is impossible whenever $p>2$.

Let $x_1^1 = T(e_1^1)$. 
Since $C_{p,q}$ is separable, we may assume that  $x_{1}^1 =   x_1^1 E_{n_1}$ for some $n_1^{(1)}:=n_1>0$.

Now, we consider $\{e_j^2\}_{j\ge 1}$. 
Since $\{e_j^2\}_{j\ge 1}$ is a weakly null sequence, it follows from \eqref{disappear} and \cite[Theorem 1.3.9]{AK} that  we may assume that there exist
$e_{j_1^{(2)}}^2, e_{j_2^{(2)}}^2$, $n_2^{(2)}>n_1^{(2)}>n_1^{(1)}$    
such that 
$$T(e_{j_1^{(2)}}^2)= x_{1}^2 = (E_{n_1^{(2)}}-E_{n_1})  x_1^2 (E_{n_1^{(2)}}-E_{n_1}) $$
and  
$$T(e_{j_2^{(2)}}^2)= x_{2}^2 = (E_{n_2^{(2)}}-E_{n_2^{(2)}})  x_2^2 (E_{n_2^{(2)}}-E_{n_2^{(2)}}) .$$

Constructing inductively, for each $k$, we have 
\begin{align}\label{disjoint}
T(e_{j_1^{(k)}}^k)= x_{1}^k =  (E_{n_1^{(k)}}-E_{n_{k-1}^{(k-1)}})    x_1^k (E_{n_1^{(k)}}-E_{n_{k-1}^{(k-1)}}) , \nonumber\\
T(e_{j_2^{(k)}}^k)= x_{2}^k = (E_{n_2^{(k)}}-E_{n_1^{(k)}})  x_2^k (E_{n_2^{(k)}}-E_{n_1^{(k)}}),\nonumber\\
 \cdots  ,\nonumber\\
T(e_{j_k^{(k)}}^k)= x_{k}^k =  (E_{n^{(k)}}-E_{n_{k-1}^{(k)}}) x_k^k (E_{n^{(k)}}-E_{n_{k-1}^{(k)}})
.\end{align}
This shows that 
$$\{e_1^1,e^2_1,e^2_2,\cdots, e_1^k,\cdots, e_k^k,\cdots \}$$
is equivalent to a sequence $
\{T(e_1^1),T(e^2_1),T(e^2_2),\cdots, T(e_1^k),\cdots, T(e_k^k),\cdots \}
$ of   mutually orthogonal self-adjoint elements in $T_{p,q}$.
Moreover, it is isometrically isomorphic to a sequence of disjointly supported elements in $\ell_{p,q}$\cite[Proposition 3.3]{HSS}.
However, since $\norm{T(e^n_j)}_{B(\cH)}\le \norm{T(e^n_j)}_{C_{p,q}}\le \norm{T}$ 
and $\chi_{_{A^k_j}}=\norm{\chi_{(0,\frac1k)}}_{L_{p,q}(0,1)} e_j^k$, it follows from \eqref{disjoint} that
\begin{align*}
\norm{T\left(
\sum_{j=1}^k \chi_{_{A_j^k}}\right) }_{B(\cH)}  &=\norm{ \norm{\chi_{(0,\frac1k)}}_{L_{p,q}(0,1)} T\left(
\sum_{j=1}^k e_j^k\right)}_{B(\cH)} \\
&=  
\norm{\chi_{(0,\frac1k)}}_{L_{p,q}(0,1)}\max_{j=1}^k \norm{  T\left(
  e_j^k\right)}_{B(\cH)} \\
&\le \norm{\chi_{(0,\frac1k)}}_{L_{p,q}(0,1)}\norm{T}
\to 0
\end{align*}
as  $n\to \infty$.
By Lemma \ref{2.1} and \cite[Proposition 3.3]{HSS}, passing to a subsequence if necessary, 
$\left\{
T\left(
\sum_{j=1}^k \chi_{_{A_j^k}}\right)  \right\}_{k\ge 1}$ in $C_{p,q}$ is equivalent to the natural vector basis of $\ell_q$.
However,since the measure of $A_j^k$ is $\frac{1}{k}$, it follows that $\cup_{j=1}^k A_{j}^k$ has measure $1$. Therefore, 
$\left\{\sum_{j=1}^k \chi_{_{A_j^k}}\right\}_{k\ge 1}$ in $U_{p,q}$ is $1$-equivalent to the natural vector basis of $\ell_{p,q}$, which is a contradiction.

Note that $U_{p,q}$ is a complemented subspace of $L_{p,q}(0,\infty)$\cite{HS21}.
If $p>2$, then $L_{p,q}(0,\infty)\not\hookrightarrow C_{p,q}$ follows immediately from the first assertion.  
When $p\le 2$, $(p,q)\ne (2,2)$, it was shown in 
 \cite[Theorem 6.2]{AHS} that $L_{p,q}(0,1)\not\hookrightarrow C_{p,q}$.
Since $L_{p,q}(0,1)$ is a complemented subspace of $L_{p,q}(0,\infty)$, 
it follows that $L_{p,q}(0,\infty )\not\hookrightarrow C_{p,q}$.
The proof is complete. 
\end{proof}
 \begin{corollary}\label{cor:app}
  Let $1<p<\infty$, $1\le q<\infty$, $p\ne q$, and let $\cM$ be a type $I$  von Neumann algebra on a separable Hilbert space equipped with a semifinite faithful normal trace $\tau$.
  If $\cM$ is atomless or atomic with all atoms having the same trace, then 
  \begin{enumerate}
    \item[(a)] if  $\cM $ is $*$-isomorphic to $\oplus_{1\le n<\infty }(\mathbb{M}_n \otimes A_n)$, where for any $N>0$, there exists $n>N$ such that  $A_n$ is a non-trivial  commutative atomless von Neumann algebra for any $n\ge 1$,
then   
        the Banach space structure of  $L_{p,q}(\cM,\tau)$ depends on the choice  of the trace $\tau$  on  $\cM =\oplus_{1\le n<\infty }(\mathbb{M}_n \otimes A_n)$;  
    \item[(b)] if the underlying algebra $\cM $ is not $*$-isomorphic to $\cM =\oplus_{1\le n<\infty }(\mathbb{M}_n \otimes A_n)$, where for any $N>0$, there exists $n>N$ such that  $A_n$ is a non-trivial  commutative atomless von Neumann algebra for any $n\ge 1$, then 
  $L_{p,q}(\cM,\tau)$ is isomorphic to one of the following mutually non-isomorphic Banach space:
  $$\ell_{p,q}^n,~ 0\le n\le \infty,  ~L_{p,q}(0,1), ~L_{p,q}(0,\infty), ~ S_{p,q}, ~C_{p,q}, 
~L_{p,q}(B(\cH )\otimes L_\infty (0,1)).$$
  \end{enumerate}
   
\end{corollary}
\begin{proof}
The underlying von Neumann algebra is $*$-isomorphic to one of the following:
\begin{enumerate}
\item  $\cM=\oplus _{1\le n\le N}(\mathbb{M}_n\otimes \cA_n)$, $N<\infty $, where $\cA_n$ is a finite-dimensional atomic commutative  von Neumann algebra with all atoms in $\cM$ having the same trace.  In this case, $L_{p,q}(\cM,\tau)\approx \ell_{p,q}^{m}$, where $m=\sum_{n=1}^N n^2 |\cA_n |$.
\item an infinite dimensional algebra $\cM=\oplus _{1\le n\le N}(\mathbb{M}_n\otimes \cA_n)$, $N<\infty $, where $\cA_n$ is an atomic commutative  von Neumann algebra with all atoms in $\cM$ having the same trace. In this case, $L_{p,q}(\cM,\tau)\approx \ell_{p,q}$.
\item $\cM=\oplus_{1\le n<\infty}(\mathbb{M}_n\otimes \cA_n)$, where 
$\cA_n$ is an atomic commutative  von Neumann algebra with all atoms in $\cM$ having the same trace,  and for any $N>0$, there exists an $n>N$ such that $\cA_n\ne \{0\}$. In this case, $L_{p,q}(\cM,\tau)\approx S_{p,q}$.
\item $\cM=\oplus_{1\le n<  \infty}(\mathbb{M}_n\otimes \cA_n) \bigoplus  (B(\cH)\otimes \cA_\infty )$, where   $\cA_n$ is an atomic commutative  von Neumann algebra,  $\cA_\infty\ne \{0\}$  and atoms in $\cM$ having the same trace. 
 We claim that $$L_{p,q}( (B(\cH)\otimes \cA_\infty ))\approx C_{p,q} .$$ Indeed, 
    we have $L_{p,q}( (B(\cH)\otimes \ell_\infty ))\hookrightarrow_c C_{p,q}$, i.e., there exists a Banach space $C $ such that $C_{p,q}\approx L_{p,q}( (B(\cH)\otimes\ell_\infty ))\oplus  C $.
Therefore, 
\begin{align*}
C_{p,q } & \approx L_{p,q}( (B(\cH)\otimes \ell_\infty ))\oplus C  \\
&\approx L_{p,q}( (B(\cH)\otimes \ell_\infty ))\oplus L_{p,q}( (B(\cH)\otimes \ell_\infty ))\oplus C  \\
&\approx L_{p,q}( (B(\cH)\otimes \ell_\infty ))\oplus C_{p,q}\\
&\approx L_{p,q}( (B(\cH)\otimes \ell_\infty )).
\end{align*}
Hence, if $|\cA_\infty|<\infty $, then  \begin{align*}
L_{p,q}( (B(\cH)\otimes \cA_\infty ))&\approx \oplus^{|\cA_\infty|}C_{p,q} \\
&\approx \oplus^{|\cA_\infty|} L_{p,q}( (B(\cH)\otimes \ell_\infty )) \\
&\approx L_{p,q}( (B(\cH)\otimes \ell_\infty ))\approx C_{p,q}.
\end{align*}
    In particular, we have $C_{p,q}\approx C_{p,q}\oplus C_{p,q}$. 
On the other hand, 
$L_{p,q}(\cM,\tau)\approx L_{p,q}( (B(\cH)\otimes \cA_\infty ))\oplus L_{p,q}(\oplus_{1\le n<\infty}\mathbb{M}_n\oplus \cA_n)$
is isomorphic to
\begin{enumerate}
  \item $L_{p,q}( (B(\cH)\otimes \cA_\infty ))\oplus \ell_{p,q}^m\approx C_{p,q }\oplus \ell_{p,q}^m$ for some natural number $m$ if $\oplus_{1\le n<\infty}\mathbb{M}_n\oplus \cA_n$ satisfies the conditions in case (1);
  \item $L_{p,q}( (B(\cH)\otimes \cA_\infty ))\oplus \ell_{p,q } \approx C_{p,q} \oplus \ell_{p,q } $ if $\oplus_{1\le n<\infty}\mathbb{M}_n\oplus \cA_n$ satisfies the conditions in case (2);
  \item $L_{p,q}( (B(\cH)\otimes \cA_\infty )) \oplus S_{p,q}\approx C_{p,q}\oplus S_{p,q} $ if $\oplus_{1\le n<\infty}\mathbb{M}_n\oplus \cA_n$ satisfies the conditions in case (3);
\end{enumerate}Note that 
$C_{p,q}\oplus \ell_2 \approx C_{p,q}$ and therefore, 
$C_{p,q}\oplus  \ell_{p,q}^m \approx C_{p,q}\oplus \ell_2 \oplus \ell_{p,q}^m \approx  C_{p,q}\oplus \ell_2 \approx  C_{p,q} $.
It is clear that if $X:= C_{p,q}\oplus \ell_{p,q  }$ or  $C_{p,q}\oplus S_{p,q}$, then $X\approx X\oplus X$.  
 Now, applying Pelczynski's decomposition technique\cite[Theorem 2.2.3]{AK}, we obtain that $L_{p,q}(\cM,\tau)\approx C_{p,q}$.
\item $\cM=\oplus _{1\le n<N}(\mathbb{M}_n\otimes \cA_n)$, where $\cA$ is atomless. If the trace on $\cM$ is finite, then $L_{p,q}(\cM,\tau)\approx L_{p,q}(0,1)$; if the trace is infinite, then  $L_{p,q}(\cM,\tau)\approx L_{p,q}(0,\infty)$;
\item  $\cM =\oplus_{1\le n<\infty }(\mathbb{M}_n \otimes A_n)$, where for any $N>0$, there exists $n>N$ such that  $A_n$ is a non-trivial  commutative atomless von Neumann algebra for any $n\ge 1$.
 Assume that $\tau=\oplus_{1\le n <\infty }{\rm Tr}_n\otimes \tau_n$, where ${\rm Tr}_n$ stands for the standard trace on $\mathbb{M}_n$ and $\tau_n$ is a tracial state. 
Then, we have $\ell_{p,q}\xhookrightarrow{c } L_{p,q}(\cM,\tau)$. 
This together with Theorem \ref{main theorem} shows that $$L_{p,q}(\cM,\tau)\not\hookrightarrow L_{p,q}(\cM,\tau'),$$
whenever $\tau'$ is a faithful normal tracial state on $\cM$. 
\item $\cM=\oplus_{1\le n<  \infty}(\mathbb{M}_n\otimes \cA_n) \bigoplus  (B(\cH)\otimes \cA_\infty )$,  where   $\cA_n$, $1\le n\le \infty $, is an atomless commutative  von Neumann algebra and  $\cA_\infty\ne \{0\}$.
Observe that 
$$L_{p,q}(\cM,\tau)\approx L_{p,q}(\cM,\tau)\oplus  L_{p,q}(\cM,\tau),$$
$$~L_{p,q}(B(\cH )\otimes L_\infty (0,1))\approx ~L_{p,q}(B(\cH )\otimes L_\infty (0,1)) \oplus ~L_{p,q}(B(\cH )\otimes L_\infty (0,1))$$
and $$L_{p,q}(\cM,\tau)\xhookrightarrow{c}L_{p,q}(B(\cH )\otimes L_\infty (0,1)),~L_{p,q}(B(\cH )\otimes L_\infty (0,1)) \xhookrightarrow{c} L_{p,q}(\cM,\tau).$$
By Pelczynski's decomposition technique\cite[Theorem 2.2.3]{AK}, we obtain that
$$L_{p,q}(B(\cH )\otimes L_\infty (0,1)) \approx L_{p,q}(\cM,\tau). $$
\end{enumerate}
 
Statement (a) is proved in case (6) above.
Below, we prove that the Banach spaces listed in statement (b) are mutually non-isomorphic. 
It is shown in \cite{KS,SS} that $\ell_{p,q}^n,~ 0\le n\le \infty,  ~L_{p,q}(0,1), ~L_{p,q}(0,\infty)$ are not isomorphic to each other. 

  Recall  a result due to Lewis\cite[Theorem]{Lewis} (see also \cite{GL74} or  \cite[Theorem 2.2]{KP70} due to Kwapien and Pelczynski):
if $C_E $ is a symmetric operator ideal in $B(\cH)$ with local unconditional structure, then the symmetric sequence space $\ell_E$ generating $C_E$ coincides with $\ell_2$.
However, by \cite[Theorem 2.6]{Gillespie}, any subspace of  $L_{p,q}(0,\infty )$ has the property that the Boolean algebra of projections generated by any pair of commuting bounded Boolean algebras of projections on  $L_{p,q}(0,\infty )$  is itself bounded.
This implies that $S_{p,q}$ (therefore, $C_{p,q}$ and $L_{p,q}(B(\cH)\otimes L_\infty(0,1))$) is not isomorphic to  a subspace of $L_{p,q}(0,\infty)$, $L_{p,q}(0,1)$ or $\ell_{p,q}$.

The main result in \cite{HSS} shows that $C_{p,q}\not\hookrightarrow S_{p,q}$ (therefore, $L_{p,q}(B(\cH)\otimes L_\infty (0,1))\not\hookrightarrow S_{p,q}$).
Therefore, $$\ell_{p,q}^n,~ 0\le n\le \infty,  ~L_{p,q}(0,1), ~L_{p,q}(0,\infty), ~ S_{p,q}, ~C_{p,q}$$ are mutually non-isomorphic.

By Proposition  \ref{LpqintoCpq}, we obtain that $$L_{p,q}(0,\infty)\not\hookrightarrow C_{p,q}.$$
Therefore, $L_{p,q}(B(\cH)\otimes L_\infty(0,1))\not\hookrightarrow  C_{p,q}$, which completes the proof for statement (b).
\end{proof}

Let $\cM$ be a von Neumann algebra equipped with a semifinite faithful normal tracial state $\tau$, which is not 
 of the form $\oplus_{1\le k <n}\mathbb{M}_k \otimes \cA_k$, $n<\infty$ ($\cA_k$ is a commutative algebra.
In the rest of this section, we complete proof for the tree claimed in the introduction. 
Note that, in the proof Corollary \ref{cor:app}, we have established  $S_{p,q}\not\hookrightarrow L_{p,q}(0,\infty)$. 
Therefore, we have 
$$  L_{p,q}(\cM,\tau)\not\hookrightarrow L_{p,q}(0,\infty ), L_{p,q}(0,1), L_{p,q}(0,1)\oplus \ell_{p,q}, \ell_{p,q}, L_{p,q}(0,1)\oplus U_{p,q}, U_{p,q}. $$
To complete the proof for  the tree in the introduction, it suffices to prove that 
$ U_{p,q}\not  \hookrightarrow  L_{p,q}(\cM,\tau) \oplus \ell_{p,q} $ and $L_{p,q}(0,\infty)\not\hookrightarrow L_{p,q}(\cM,\tau)\oplus U_{p,q}$.

\begin{corollary}\label{cor5.3}
   Let $ (p,q)\in (1,\infty)\times [1,\infty )$ and let $ \cM $ be a von Neumann algebra equipped with a finite faithful normal trace $\tau$.
If $p\neq q,$ then $$ U_{p,q}\not  \hookrightarrow  L_{p,q}(\cM,\tau) \oplus \ell_{p,q}.  $$ 
\end{corollary}
\begin{proof}
 
It suffices to consider the case when $\cM$ is infinite-dimensional and atomless. 
Assume, by way of  contradiction, that there exists an isomorphism $T$ from $U_{p,q}$ into $L_{p,q}(\cM,\tau) \oplus \ell_{p,q}$. 
Without loss of generality, we may assume that 
$\{A_n\}=\cup_n B_n$ and 
$U_{p,q}=L_{p,q}\{ \{ A_n\}_{n\ge 1}\}$, where for each $n\ge 1$,     $B_n=\{A_k^n\}_{k=1}^\infty$ is    a set of infinite many atoms of measure $\frac1n$.

Let  $\left\{e_{j}^n\right\}_{n,j\ge 1}$ be the normalized basis spanned by   $\left\{\chi_{_{A_j^n}}\right\}_{n,j\ge 1}$, the set of    characteristic functions of $A_j^n$'s.
Let $P_1$ and $P_2$ be  projections  from $L_{p,q}(\cM,\tau)\oplus \ell_{p,q} $ onto $L_{p,q}(\cM,\tau)\oplus 0$ and $ 0\oplus \ell_{p,q}$, respectively.

We claim that for any $n\ge 1$ and  $1\le n_1 <n_2 < \cdots <n_n$, 
\begin{align}\label{P2TENJ}
\norm{P_2T \left(  \norm{\chi_{(0,\frac1n)}}_{L_{p,q}(0,1)}  \sum_{1\le j\le n}e_{n_j}^n  \right)}_{\ell_{p,q}}\not \to 0
\end{align}
as $n\to \infty$.
Indeed, otherwise, passing to a subsequence of $\{n\}_{n\ge 1}$ if necessary, we have
$$ \left\{ e_n^{\ell_{p,q}} \right \}_{n\ge  1}
\sim
\left\{ \norm{\chi_{(0,\frac1n)}}_{L_{p,q}(0,1)}  \sum_{1\le j \le n } e_{n_j}^n    \right\}_{n\ge  1}
\sim
\left \{  P_1T  \left(
 \norm{\chi_{(0,\frac1n)}  }_{L_{p,q}(0,1)} \sum_{1\le j \le n } e_{n_j}^n   \right)
  \right\}_{n\ge  1}.$$
  This implies that $$\ell_{p,q}
   \hookrightarrow  L_{p,q}(\cM,\tau ), $$
which yields  a contradiction to the Theorem \ref{main theorem},    $\ell_{p,q}\not \hookrightarrow L_{p,q}(\cM,\tau )$.
This proves \eqref{P2TENJ}.
Hence, passing to a subsequence if necessary, there exists a $\delta>0$ such that 
\begin{align}\label{P2TENJ2}
\norm{P_2T \left(   \sum_{1\le j\le n}e_{n_j}^n  \right)}_{\ell_{p,q}}\ge \frac{\delta}{ \norm{\chi_{(0,\frac1n)}}_{p,q}} 
\end{align}
for all $n\ge 1$. 

If, for infinitely many   $n $'s, $$ \liminf_{j\to \infty} \norm{ P_2T(e^n _j)}_{\ell_{p,q} }\to 0, $$
then, passing to a subsequence of $\{n\}$ if necessary, we may assume that 
for all  $n \ge 1 $, $$ \liminf_{j\to \infty} \norm{ P_2T(e^n _j)}_{\ell_{p,q} }\to 0. $$
For each $n\ge 1$, passing to a subsequence of $\{j\}$ if necessary, we may assume that 
$$ \norm{ P_2T(e^n _j)}_{\ell_{p,q} }\to 0$$
as $j\to \infty $. 
Denote by $K$ the basis constant of the basic sequence $\left\{T(e^{n}_j)\right\}_{n,j\ge 1}$. 
Let $e^1_{n_1}$ be such that $$\norm{P_1 T(e^1_{n_1})  -T(e^1_{n_1}) }_{L_{p,q}(\cM,\tau)\oplus \ell_{p,q}}\le \norm{   P_2T(e^n _j)}_{L_{p,q}(\cM,\tau)\oplus\ell_{p,q}} < \frac{\norm{T^{-1}}}{4 K}.
$$ 
Let $n_2>n_1$ be such that 
$$\norm{P_1 T(e^1_{n_2})  - T(e^1_{n_2}) }_{L_{p,q}(\cM,\tau) \oplus \ell_{p,q}} < \frac{\norm{T^{-1}}}{2^3 K}.$$ 
Let $n_3>n_2$ be such that 
$$\norm{P_1 T(e^2_{n_3})   - T(e^2_{n_3}) }_{L_{p,q}(\cM,\tau) \oplus \ell_{p,q}} < \frac{\norm{T^{-1}}}{2^4 K}.$$
Constructing inductively, we obtain a  sequence $\left\{e^1_{n_1},e^1_{n_2},e^2_{n_3},e^1_{n_4},e^2_{n_5},e^3_{n_6},\cdots \right\} $ with $n_1<n_2<\cdots$ 
such that 
$$\norm{P_1 T(e_{n_j}^m )  - T(e_{n_j}^m ) }_{L_{p,q}(\cM,\tau) \oplus \ell_{p,q}} < \frac{\norm{T^{-1}}}{2^{n+1} K}.$$
The choice of the sequence   $\{e_{n_j}^m  \}$ of $\{e_j^n\}$ is shown in the following matrix:
\begin{equation*}
\begin{matrix}
n=1&e_{n_1}^1 & \cdots &  e_{n_2}^1   &  \cdots  &  e_{n_4}^1  & \cdots&\cdots  & e_{n_7}^1 &\cdots&\cdots&  \cdots&   e_{n_{11}}^1&  \cdots&  \cdots\\
n=2&\cdots  & \cdots & \cdots  &  e_{n_3}^2   & \cdots &  e_{n_5}^2&\cdots &\cdots  & e_{n_8}^2&\cdots&  \cdots &  \cdots&  e_{n_{12}}^2&  \cdots\\
n=3&\cdots & \cdots & \cdots    & \cdots & \cdots &  \cdots &  e_{n_6}^3 & \cdots & \cdots   & e_{n_9}^3  &  \cdots&  \cdots&  \cdots
\\
n=4&\cdots & \cdots & \cdots    & \cdots & \cdots &  \cdots &  \cdots & \cdots & \cdots   &  \cdots &  e_{n_{10}}^4&  \cdots &  \cdots
\\
n=5&\cdots & \cdots & \cdots    & \cdots & \cdots &  \cdots &  \cdots & \cdots & \cdots   &  \cdots &   \cdots &  \cdots&  \cdots
\end{matrix}.
\end{equation*} 
Since $ \left\{e^1_{n_1},e^1_{n_2},e^2_{n_3},e^1_{n_4},e^2_{n_5},e^3_{n_6},\cdots \right\} $ is $1$-equivalent to   $\left\{e_j^n\right\}_{n,j\ge1}$ in $U_{p,q}$, we still denote 
 $\{e_{n_j}^n  \}$ by $\{e_j^n\}$
for simplicity.
Since  $ 2K \sum_{j=1}^\infty \frac{\frac{\norm{T^{-1}}}{2^{j+1}K}  }{\norm{T^{-1}}}=\frac12<1$, it follows from \cite[Theorem 1.3.9]{AK} that   the unconditional basic sequence $\left\{e_{j}^n\right\}_{n,j\ge 1}$ in $U_{p,q}$ is equivalent to 
$$   \left\{P_1T(e_{j}^n)\right\}_{n,j\ge 1}   $$
in $L_{p,q}(\cM,\tau)$. 
This implies that $U_{p,q} \hookrightarrow L_{p,q}(\cM,\tau)$, which is impossible due to Theorem \ref{main theorem}.

Passing to a subsequence if necessary, we may assume that
for each $n$,   there exists $\delta_n>0$ such that
$$ \liminf_{j\to \infty} \norm{ P_2T(e^n _j)}_{\ell_{p,q} }>  \delta_n. $$
Without loss of generality, we may assume that $\norm{ P_2T(e^n _j)}_{\ell_{p,q} }>  \delta_n$ for all $j\ge 1$. 
Let \begin{align}\label{fnjp2}
f^n_j := \frac{P_2T (e^n_j)}{\norm{P_2T (e^n_j) }_{\ell_{p,q}}}.
 \end{align}In particular, $\norm{f_j^n}_{\ell_{p,q}}=1$ and therefore, $\norm{f_{j}^n}_{\ell_\infty} \le 1 $.
For $f_{j_1^1}:=f^1_1$ and any $\varepsilon_1>0$, since $\ell_{p,q}$ is separable, there exists a finitely supported element $g_1^1:= f_{j_1^1}  s(g_1^1) $ (where $s(g_1^1)$ is the support of $g_1^1$) such that
$\norm{f_{j_1^1} -g_1^1 }_{\ell_{p,q} }\le \varepsilon_1.$
In particular, $\norm{g_{1}^1}_\infty \le 1 $.

Let $\varepsilon_2>0$.
Since $\{e^2_j\}_{j\ge 1}$ is weakly null,  it follows that    $P_2 T(e^2_j)\to_j 0$ weakly. 
Hence, $(P_2 T(e^2_j) ) s(g_1^1 ) \to_j 0$ in $\norm{\cdot}_{\ell_{p,q}}$.
By the separability of $\ell_{p,q}$,  
  there exist  $f_{j^2_1}\in \{T(e^2_j)\}_{j\ge 1}$ and  a finitely supported element $g_1^2:= f_{j_1^2}  s(g_1^2) $ (where $s(g_1^2)$ is the support of $g_1^2$) such that
 $\norm{f_{j^2_1} -g^2_1}_{\ell_{p,q}}\le \varepsilon_2 $
and $g_1^1 g_1^2=0$.

Let $\varepsilon_3>0$.
Since $\{e^2_j\}_{j\ge 1} $ is weakly null, it follows that there exist  $f_{j_2^2}\in \{T(e^2_j)\}_{j\ge 1}$ and   a finitely supported element $g_2^2:= f_{j_2^2}  s(g_2^2) $ (where $s(g_2^2)$ is the support of $g_2^2$) such that
 $\norm{f_{j_2^2} -g^2_2}_{\ell_{p,q}}\le \varepsilon_3 $
and $g_1^1 g_2^2=0=g_1^2g_2^2 $.

Since $\varepsilon_i$'s can be chosen to be arbitrarily small,  we may construct inductively a sequence $\{g^n_i\}_{1\le i\le n,n\ge 1 }$ of disjointly supported elements in $\ell_{p,q}$  which is  equivalent to $\{f_{j^n_i}\}_{1\le i\le n,n\ge 1}$\cite[Theorem 1.3.9]{AK}.
 To lighten the  notions, we may assume that   $\{f_{j^n_i}\}_{1\le i\le n,n\ge 1}=\{f_{ i}^n
=T(e_j^n)
\}_{1\le i\le n,n\ge 1}$. 
Since $g_{i}^n$ are disjointly supported and $\norm{g_{i}^n}_\infty \le 1$, it follows
that there exists a constant ${\rm Const} >0$ (independent of $n$) such that 
\begin{align*}
{\rm Const} \cdot  \norm{
\frac{
    \sum_{1\le i \le n}  \norm{P_2T  (e_{i}^n) }_{\ell_{p,q}} g_i^n  }
    {
    \norm{
   \sum_{1\le i \le n}  \norm{P_2T  (e_{i}^n) }_{
\ell_{p,q}} g_i^n }_{\ell_{p,q}}
    }
 }_{\ell_\infty}
&\le 
\norm{
\frac{
    \sum_{1\le i \le n}  \norm{P_2T  (e_{i}^n) }_{\ell_{p,q}} g_i^n  }
    {
    \norm{
   \sum_{1\le i \le n}  \norm{P_2T  (e_{i}^n) }_{\ell_{p,q}} f_{i}^n }_{\ell_{p,q}}
    }
 }_{\ell_\infty}\\
 &\stackrel{\eqref{fnjp2}}{=}
 \norm{
\frac{
    \sum_{1\le i \le n}  \norm{P_2T  (e_{ i}^n) }_{\ell_{p,q}} g_i^n  }
    {
    \norm{
   \sum_{1\le i \le n}   P_2T  (e_{ i}^n)  }_{\ell_{p,q}}
    }
 }_{\ell_\infty}\\
 &
\le 
\frac{\max _{1\le i \le n} \left\{   \norm{P_2T  (e_{i}^n) }_{\ell_{p,q}}  \right\} }
    {
    \norm{
   \sum_{1\le i \le n}   P_2T  (e_{i}^n)  }_{\ell_{p,q}}
    }
\\
 &
\stackrel{\eqref{P2TENJ2}}{ \lesssim }\norm{P_2T} \frac{ \norm{\chi_{(0,\frac1n)}}_{L_{p,q(0,1)}  }  }{ \delta  } \to  0
\end{align*}
as $n \to \infty$.

By Lemma \ref{2.1}, passing to a subsequence of $\{n\}$ if necessary, 
we may assume that   \begin{align}\label{gequ}\left\{\frac{
    \sum_{1\le i \le {n}}  \norm{P_2T  (e_{i }^{n}) }_{\ell_{p,q}} g_i^{n}  }
    {
    \norm{
   \sum_{1\le i \le {n}}  \norm{P_2T  (e_{i}^{n}) }_{\ell_{p,q}} g_i^{n} }_{\ell_{p,q}}
    } \right\}_{n\ge 1} \sim \left\{e_{k}^{\ell_q}\right \}_{n\ge 1}
 .\end{align}
Recall that $  e_{i}^{n } = \frac{1}{\norm{ \chi_{_{A_i^n}}}_{U_{p,q}}} \chi_{_{A_i^n}}  $. We have $ \norm{\chi_{(0,\frac1n)}}_{L_{p,q(0,1)}  } \sum_{1\le i \le {n }}   e_{i}^{n} = \chi_{_{\cup_{1\le i \le {n}}  A^{n }_i  }}  $. 
Since   $A_n^i$, $i,n\ge 1$, are disjointly supported and the measure of $\cup_{1\le i\le {n}}  A^{n }_i$ is $1$, it follows that 
\begin{align*}
  \left\{e_{n}^{\ell_{p,q}}\right\}_{n\ge 1} 
 \sim   \left\{ \norm{\chi_{(0,\frac{1}{n})}}_{L_{p,q(0,1)}  } \sum_{1\le i \le {n }}   e_{i}^{n }\right\}_{n\ge 1}
 \end{align*}
 in $U_{p,q}$, which is equivalent to 
\begin{align*}&\qquad   \left \{T\left (\norm{\chi_{(0,\frac1n)}}_{L_{p,q(0,1)}  } \sum_{1\le i \le {n }} e_{i}^{n }\right ) \right\}_{n\ge 1} \\
&~=\left\{  P_1T\left (\norm{\chi_{(0,\frac1n)}}_{L_{p,q(0,1)}  } \sum_{1\le i\le {n }} e_{i}^{n }\right ) \oplus P_2T\left (\norm{\chi_{(0,\frac1n)}}_{L_{p,q(0,1)}  } \sum_{1\le i \le {n }} e_{i}^{n }\right ) \right\}_{n\ge 1} \\
&\stackrel{\eqref{fnjp2} }{=}\left\{  P_1T\left (\norm{\chi_{(0,\frac1n)}}_{L_{p,q(0,1)}  } \sum_{1\le i \le {n }} e_{i}^{n }\right ) \oplus \norm{\chi_{(0,\frac1n)}}_{L_{p,q(0,1)}  }  \sum_{1\le i \le {n }}  \norm{P_2T  (e_{i}^{n}) }_{\ell_{p,q}}  f_{i}^{n }  \right\}_{n\ge 1}
\end{align*} in
 $L_{p,q}(0,1) \oplus \ell_{p,q}$.  
 By \eqref{P2TENJ}, 
we may assume that 
  $$\inf_{n\ge 1} \norm{\norm{\chi_{(0,\frac1n)}}_{L_{p,q}(0,1)}\sum_{1\le i  \le {n }}  \norm{P_2T  (e_{i}^{n}) }_{\ell_{p,q}}  f_{i}^{n }}_{\ell_{p,q}} >  0.$$ 
Also, since $\{g_{i}^n\}_{1\le i\le n, n\ge 1}\sim \{f_{i}^n\}_{1\le i\le n, n\ge 1}$, it follows  that 
$$  
\left\{e_{n}^{\ell_{p,q}}\right\}_{n\ge 1}
 \sim  \left\{  P_1T\left (\norm{\chi_{(0,\frac1n)}}_{L_{p,q(0,1)}  } \sum_{1\le i \le {n }} e_{i}^{n }\right ) \oplus \frac{\sum_{1\le i \le n }   \norm{P_2T  (e_{i}^{n}) }_{\ell_{p,q}}   g_i^{n }}{\norm{\sum_{1\le i \le {n }}  \norm{P_2T  (e_{i}^{n}) }_{\ell_{p,q}}   g_i^{n }}}_{\ell_{p,q}} \right\}_{n\ge 1} . 
 $$
 Hence, by \eqref{gequ},  for any $ \alpha =(\alpha_1,\alpha_2,\cdots,\alpha_n,\cdots)$, we have 
\begin{align*}&\quad \norm{\alpha }_{\ell_{p,q}}\\
&\sim 
 \left\| \left(  \alpha_n P_1T\left (\norm{\chi_{(0,\frac1n)}}_{L_{p,q(0,1)}  } \sum_{1\le i \le {n }} e_{i}^{n }\right ) 
\oplus 
 \alpha_n \frac{\sum_{1\le i \le n }   \norm{P_2T  (e_{i }^{n}) }_{\ell_{p,q}}   g_i^{n }}{\norm{\sum_{1\le i \le {n }}  \norm{P_2T  (e_{i }^{n}) }_{\ell_{p,q}}   g_i^{n }}}_{\ell_{p,q}} \right) _{n\ge 1}  \right\|_{L_{p,q}(\cM,\tau)\oplus \ell_{p,q}}\\
 &= \left\| \left(  \alpha_n P_1T\left (\norm{\chi_{(0,\frac1n)}}_{L_{p,q(0,1)}  } \sum_{1\le i \le {n }} e_{i}^{n }\right ) \right)_{n\ge 1}  \right\|_{L_{p,q}(\cM,\tau)}\\
 &\qquad +
 \left\| \left(   \alpha_n \frac{\sum_{1\le i \le n }   \norm{P_2T  (e_{i}^{n}) }_{\ell_{p,q}}   g_i^{n }}{\norm{\sum_{1\le i \le {n }}  \norm{P_2T  (e_{i}^{n}) }_{\ell_{p,q}}   g_i^{n }}}_{\ell_{p,q}} \right)_{n\ge 1}  \right\|_{ \ell_{p,q}}\\
  &\stackrel{\eqref{gequ}}{\sim} \left\| \left(  \alpha_n P_1T\left (\norm{\chi_{(0,\frac1n)}}_{L_{p,q(0,1)}  } \sum_{1\le i \le {n }} e_{i}^{n }\right ) \right)_{n\ge 1}  \right\|_{L_{p,q}(\cM,\tau)}  +
 \left\| \left(   \alpha_ n  \right)_{n\ge 1}  \right\|_{ \ell_{ q}}\\
 &= \left\| \left( \alpha_n \left(   P_1T\left (\norm{\chi_{(0,\frac1n)}}_{L_{p,q(0,1)}  } \sum_{1\le i \le {n }} e_{i}^{n }\right )   \oplus     e^{\ell_q}_n  \right)\right)_{n\ge 1} \right\|_{L_{p,q}(\cM,\tau)\oplus  \ell_{ q}}
\end{align*}
We obtain that   
 \begin{align*}
\ell_{p,q}\hookrightarrow L_{p,q}(\cM,\tau)\oplus \ell_{q}
\hookrightarrow L_{p,q}(\cM,\tau)\oplus L_{p,q}(0,1 ) \approx L_{p,q}(\cM\oplus L_\infty (0,1),\tau\oplus m  ).
\end{align*}
However, 
$\cM\oplus L_\infty (0,1)  $ is a finite von Neumann algebra and $\tau\oplus m $ is  a finite faithful normal trace  on $\cM\oplus L_\infty (0,1)$.
However, by Theorem \ref{main theorem}, we have $\ell_{p,q}\not\hookrightarrow L_{p,q}(\cM\oplus L_\infty (0,1) ,\tau\oplus m )$, which 
is a contradiction.
\end{proof}


\begin{corollary}\label{infinityinto}   Let $ (p,q)\in (1,\infty)\times [1,\infty )$ and let $ \cM $ be a von Neumann algebra equipped with a finite faithful normal trace $\tau$.
If $p\neq q,$ then
$$L_{p,q}(0,\infty)\not\hookrightarrow L_{p,q}(\cM,\tau) \oplus U_{p,q}.$$
\end{corollary}
\begin{proof} 
Assume, by way of  contradiction, that there exists an isomorphic embedding
 $$T: L_{p,q}(0,\infty)  \hookrightarrow L_{p,q}(\cM,\tau) \oplus U_{p,q}.$$
 
 For every $n,k\in \mathbb{N}$, $t\in (0,\infty)$, we define
$$
r_{n,k}(t):=
\begin{cases}r_n(t-k+1), &t\in (k-1,k];\\
0, &\mbox{elsewhere},
\end{cases}
$$
 where $r_n$, $n\in \mathbb{N}$, are the Rademacher functions\cite[p.24]{LT2}.
 By \cite[Theorem 5]{KS}, $\{r_{n,k}\}_{k,n=1}^\infty $ is a basic sequence, and hence,
$\{T(r_{n,k})\}_{k,n=1}^\infty $
is also a basic sequence with a  basic constant, which does not exceed $\left\|T\right\|$.
Let $P_1 $ and $P_2$ be projections from $L_{p,q}(\cM,\tau) \oplus U_{p,q}$
onto $L_{p,q}(\cM,\tau)\oplus 0$ and $0\oplus U_{p,q}$, respectively.

 Consider the sequence $\{P_2T(r_{n,k})\}_{n=1}^\infty$, $k\in \mathbb{N}$.
Without loss of generality, we may assume that
$$
\inf_n \norm{P_2T(r_{n,k})}_{U_{p,q}} =0,~\forall k\in \mathbb{N},$$
or else
$$
\inf_n \norm{P_2T(r_{n,k})}_{U_{p,q}}
>0,~\forall k\in \mathbb{N}.$$

If $
\inf_n \norm{P_2T(r_{n,k})}_{U_{p,q}} =0~\forall k\in \mathbb{N}$, 
it is immediate that there exists a sequence $\{n\}_{k=1}^\infty$ of integers such that
$$\{T(r_{n, k})\}_{k=1}^\infty \sim
\{P_1T(r_{n, k})\}_{k=1}^\infty.$$
Hence, 
$$\{e^{\ell_{p,q}}_k\}_{k=1}^\infty
\sim
\{r_{n, k}\}_{k=1}^\infty \sim
\{T(r_{n, k})\}_{k=1}^\infty \sim
\{P_1T(r_{n, k})\}_{k=1}^\infty. $$
In particular, $\ell_{p,q}\hookrightarrow L_{p,q}(\cM,\tau)$, 
which contradicts with Theorem \ref{main theorem}.

Now suppose that $
\inf_n \norm{P_2T(r_{n,k})}_{p,q}
>0,~\forall k\in \mathbb{N}.$
Without loss of generality, we may assume that
$$\inf_k \left(
\inf_n \norm{P_2T(r_{n,k})}_{U_{p,q}}\right)=0,$$
or else
$$\inf_k \left(
\inf_n \norm{P_2T(r_{n,k})}_{U_{p,q}}\right)>0.$$

If $\inf_k \left(
\inf_n \norm{P_2T(r_{n,k})}_{U_{p,q}} \right)=0,$
then it is immediate  that
 there exist two
sequences $\{n\}_{k=1}^\infty$ and $\{j_k\}_{k=1}^\infty $
 of integers such that
$$\{e^{\ell_{p,q}}_k\}_{k=1}^\infty
\sim
\{r_{n,j_k}\}_{k=1}^\infty \sim
\{T(r_{n,j_k})\}_{k=1}^\infty \sim
\{P_1T(r_{n,j_k})\}_{k=1}^\infty. $$
In particular, $\ell_{p,q}\hookrightarrow L_{p,q}(\cM,\tau)$,
which contradicts with  $\ell_{p,q}\not\hookrightarrow L_{p,q}(\cM,\tau) $ (see Theorem \ref{main theorem}).

Now, suppose that
$\inf_k \left(
\inf_n \norm{P_2T(r_{n,k})}_{U_{p,q}} \right)>0 $.
Since the sequence
$\{P_2T(r_{n,k})\}_{n=1}^\infty $ is weakly null\cite[Proposition 2.c.10]{LT2}, without loss of generality, by \cite[Proposition 1.a.12]{LT1}, we may
assume that
for any fixed $k\in \mathbb{N}$, the sequence $\{P_2T(r_{n,k})\}_{n=1}^\infty $
is equivalent to a sequence of disjointly supported
elements in $U_{p,q}$.
By \cite[Lemma 2]{KS}, there exist two sequences
$\{a_{k}\}_{k=1}^\infty $
and $\{b_{k}\}_{k=1}^\infty $,
 and a subsequence of $\{r_{n,k}\}_{n=1}^\infty$ (for simplicity, still denoted by $\{r_{n,k}\}_{n=1}^\infty$)
such that for
$$F_k:= \frac{\sum_{n=a_k}^{b_k}P_2T(r_{n,k}) }{\norm{\sum_{n=a_k}^{b_k}P_2T(r_{n,k})}_{U_{p,q}}} \in U_{p,q},~ k\ge 1, $$
the convergence $ F_k^*\to 0$ holds as $k\to \infty $.
By Lemma \ref{2.1}, passing to a subsequence if necessary,
we may assume that
\begin{align}\label{FKlq}
\{F_{k}\}_{k=1}^\infty \sim \{e^{\ell_q}_k\}_{k=1}^\infty . 
\end{align}
On the other hand, by
\cite[Theorem 6]{KS},
for the elements
$$G_k:= \frac{\sum_{n=a_k}^{b_k}r_{n,k}}{\norm{\sum_{n=a_k}^{b_k}r_{n,k} }_{L_{p,q}(0,\infty)}}\in [r_{n,k}]_{n=1}^\infty,$$
we have
\begin{align}\label{GKlpq}
\{G_{k}\}_{k=1}^\infty \sim\{e^{\ell_{p,q} }_k \}^\infty_{k=1}.
\end{align}
Denote $\alpha_k:=
\frac{
 \norm{
 \sum_{n=a_k}^{b_k}P_2T(r_{n,k})}_{U_{p,q}}
 }{ \norm{\sum_{n=a_k}^{b_k} r_{n,k} }_{L_{p,q}(0,\infty )}}$. In particular, $\sup_k \alpha_k \le \norm{P_2T}$
and 
\begin{align}\label{P2TGK}
P_2T(G_k) =\alpha_k  F_k, ~k\ge 1. 
\end{align}
Passing to a subsequence if necessary, we  may assume that $$\alpha_k\to 0\mbox{ as } k\to \infty $$ or else $$\inf_k\alpha_k >0 .$$
If $\alpha_k\to 0$ as $k\to \infty$, then $\norm{P_2T(G_k)}_{U_{p,q}}\to 0$ and hence, passing to a subsequence if necessary,
$$ \{P_1T(G_k)\}_{k=1}^\infty
\sim
\{T(G_k) \}_{k=1}^\infty\sim
\{  G_k  \}_{k=1}^\infty\stackrel{\eqref{GKlpq}}\sim \{e^{\ell_{p,q}}_k \}_{k=1}^\infty .$$
This implies that $\ell_{p,q}\hookrightarrow L_{p,q}(\cM,\tau)$, which is impossible (see Theorem \ref{main theorem}).

Now, if $\inf_{k}\alpha_k >0$, then, passing to a subsequence if necessary, we have
$$ \{e^{\ell_q}_k\}_{k=1}^\infty \stackrel{\eqref{FKlq}}{\sim} \{F_k\}_{k=1}^\infty \stackrel{\eqref{P2TGK}}{\sim} \{P_2T (G_k) \}_{k=1}^\infty$$
 and 
 $$\{e_k^{\ell_{p,q}}\}_{k=1}^\infty
\stackrel{\eqref{GKlpq}}\sim
\{ G_k\}_{k=1}^\infty
 \sim
\{P_1T(G_k) \oplus P_2T(G_k)\}_{k=1}^\infty.$$
 For any $ \alpha =(\alpha_1,\alpha_2,\cdots,\alpha_n,0,0,\cdots)$, we have 
 \begin{align*}
   \norm{(\alpha_k)_{k\ge 1 }}_{\ell_{p,q}} & \sim \norm{(\alpha_k P_1T(G_k) \oplus \alpha_k P_2T(G_k))_{k\ge 1 }}_{L_{p,q}(\cM,\tau)\oplus U_{p,q}} \\
    & = \norm{(\alpha_k P_1T(G_k) )_{k\ge 1 }}_{L_{p,q}(\cM,\tau) }+\norm{( \alpha_k P_2T(G_k))_{k\ge 1 }}_{  U_{p,q}}\\
   & \sim \norm{(\alpha_k P_1T(G_k) )_{k\ge 1 }}_{L_{p,q}(\cM,\tau) }+\norm{( \alpha_k )_{k\ge 1 }}_{  \ell_{ q}}\\
    & =\norm{\left (\alpha_k \left ( P_1T(G_k) \oplus  e^{\ell_q}_k  \right) \right)_{k\ge 1 }}_{ L_{p,q}(\cM,\tau)\oplus  \ell_{ q}}
 \end{align*}
Therefore,   
\begin{align*}
\ell_{p,q}\hookrightarrow L_{p,q}(\cM,\tau)\oplus \ell_q 
\hookrightarrow L_{p,q}(\cM,\tau)\oplus L_{p,q}(0,1)\approx L_{p,q}(\cM\oplus [0,1],\tau\oplus m),
\end{align*}
which contradicts with Theorem \ref{main theorem}. 
 \end{proof}

\end{document}